\documentclass{amsart}

\usepackage{amssymb,amsmath,amsfonts,amsthm,mathrsfs}
\usepackage{color}\usepackage[dvipsnames]{xcolor}\usepackage{euler,palatino}
\usepackage{graphicx}
\usepackage{adjustbox}
\usepackage{mathrsfs}
\usepackage{mathtools}
\usepackage{enumitem}
\usepackage{tikz}
\usepackage{cite}
\usepackage{hyperref}
\usepackage[
top    = 1.00in,
bottom = 1.00in,
left   = 1.00in,
right  = 1.00in]{geometry}
\usepackage{tikz-cd}
\usepackage{ stmaryrd }
\usetikzlibrary{decorations.pathreplacing,decorations.markings}
\usetikzlibrary{arrows}

\tikzset{
commutative diagrams/.cd,
arrow style=tikz,
diagrams={>=latex}}

\newtheorem{thm}{Theorem}
\newtheorem{claim}{Claim}
\newtheorem{cor}{Corollary}

\newtheorem{lem}{Lemma}
\newtheorem{prop}{Proposition}

\theoremstyle{definition}
\newtheorem{defn}{Definition}

\newtheorem{notn}{Notation}

\theoremstyle{definition}

\newtheorem{remark}{Remark}

\newcommand\Om{\Omega}
\newcommand\Z{\mathbb{Z}}

\newcommand\Th{{}^{\mathrm{th}}}

\newcommand\Bcal{\mathcal{B}}

\newcommand\Abf{\mathbf{A}}
\newcommand\Bbf{\mathbf{B}}
\newcommand\Acal{\mathcal{A}}
\newcommand\Xbf{\mathbf{X}}
\newcommand\Gbf{\mathbf{G}}
\newcommand\Hbb{\mathbb{H}}
\newcommand\Pbb{\mathbb{P}}
\newcommand\Zbb{\mathbb{Z}}

\newcommand\meet{\wedge}
\newcommand\join{\vee}
\newcommand\sgn{\operatorname{sgn}}
\newcommand\Gcal{\mathcal{G}}

\newcommand{\limn}[1]{\operatorname{lim}^{#1}}
\newcommand{\dlim}{\varinjlim}

\newcommand\dom{\operatorname{dom}}

\author[N. Bannister]{Nathaniel Bannister}
\author[J. Bergfalk]{Jeffrey Bergfalk}
\author[J. Tatch Moore]{Justin Tatch Moore}

\title[On the additivity of strong homology]{On the additivity of strong homology for
locally compact separable metric spaces}



\thanks{This research was primarily completed during the summer of 2020 while the first author was an undergraduate at Cornell and
the second author was a postdoctoral fellow at the Universidad Nacional Aut\'{o}noma de M\'{e}xico campus in Morelia.
The first author's research was supported through a grant from of the Cornell math department.
The third author was partially supported by NSF grant DMS-1854367.}

\begin{document}

\maketitle

\begin{abstract}
We show that it is consistent relative to a
weakly compact cardinal that strong homology is additive
and compactly supported within the class of locally
compact separable metric spaces.
This complements work of Marde\v{s}i\'{c} and Prasolov \cite{SHINA} showing that the Continuum Hypothesis implies that a countable sum of Hawaiian earrings witnesses the failure of strong homology to possess either of these properties.
Our results build directly on work of Lambie-Hanson and
the second author \cite{SVHDL}
which establishes the consistency, relative to a weakly compact cardinal, of
$\limn{s} \Abf = 0$ for all $s \geq 1$ for a certain pro-abelian group $\Abf$; we show that that work's arguments carry implications for the vanishing and additivity of the $\limn{s}$ functors over a substantially more general class of pro-abelian groups indexed by $\mathbb{N}^{\mathbb{N}}$.
\end{abstract}

\section{Introduction}

Samuel Eilenberg and Norman Steenrod's 1952 work \emph{Foundations of Algebraic Topology} is generally credited with having supplied exactly that: a cohesive and foundational framework for the development of the field \cite{EileStee}.
Its core contribution was an axiomatic approach to homology and cohomology functors; more precisely, Eilenberg and Steenrod showed that such functors from the category of finite polyhedra are fully determined by a short list of requirements now known as the \emph{Eilenberg-Steenrod axioms}.
Equally striking, though, was their attention therein to limit constructions within the subject, chiefly as embodied in \v{C}ech homology and cohomology theories; this material occupies roughly one third of their text.
Eilenberg and Steenrod underscored the delicacy of the relations between such constructions and their axioms by showing that any homology theory commuting, like \v{C}ech homology, with inverse limits must fail to satisfy their Exactness Axiom.
This delicacy is a fundamental context for the present article.

Much of the interest of limit operations lay in the possibilities for extending what Eilenberg and Steenrod had shown to be canonical homology and cohomology functors on the category of finite polyhedra to broader classes of topological spaces.
A decisive advance in this direction was Milnor's 1961 \emph{On axiomatic homology theory}, which showed that the Eilenberg-Steenrod axioms together with one further axiom determine a unique homology theory (up to natural isomorphism) on the category $\mathsf{HCW}$ of spaces having the homotopy type of a CW complex \cite{Milnor}.
Milnor termed this axiom \emph{Additivity}; it is the requirement that a homology theory meet the conditions of the following definition.
\begin{defn}A homology theory $H_{\bullet}$ is \emph{additive on the class $\mathcal{C}$ of topological spaces} if for every natural number $p$ and every family $\{X_\alpha\,|\,\alpha\in A\}$ with each $X_\alpha$ and $\coprod_A X_\alpha$ in $\mathcal{C}$, the map
 $$i^{*}_p:\bigoplus_A H_p(X_\alpha)\rightarrow H_p(\coprod_A X_\alpha)$$
 induced by the inclusion maps $i_\alpha:X_\alpha\hookrightarrow\coprod_A X_\alpha$ is an isomorphism.
\end{defn}
Milnor noted that the most prominent additive homology functor is singular homology, but two of his contemporaneous works would simultaneously complicate this distinction.
Let $Y^{(p)}$ denote the one-point compactification of a countable sum of $p$-dimensional disks, or what is often more succinctly termed a \emph{compact wedge} or \emph{cluster of $p$-spheres} or \emph{$p$-dimensional Hawaiian earring}.
In 1961, Barratt and Milnor showed that for all $p>1$ the rational-coefficient singular homology groups of $Y^{(p)}$ are nonzero in infinitely many dimensions --- a striking instance of singular homology's failure to reliably reflect a space's \emph{shape}, in senses the field of \emph{(strong) shape theory} makes precise \cite{BarrattMilnor, Shape2, Shape1}.
That same year, Milnor showed that the Eilenberg-Steenrod axioms together with two further axioms determine a unique homology functor on the category $\mathsf{MC}$ of metric compacta \cite{Milnor2}.
One of these axioms, dual in spirit to Additivity, simply required that this functor behave productively on clusters (as singular homology does not); the most prominent incarnation of this unique theory at the time carried the name of \emph{Steenrod homology}.
Here two points are worth remarking:
\begin{itemize}
\item \v{C}ech homology satisfies all three of Milnor's supplementary axioms referenced above.
These virtues stem from a strategy for extending homology functors beyond the class of polyhedra via \emph{systems of polyhedral approximations} to a given space.\footnote{The nerves of the open covers of a space form the best-known such system.
See \cite{Ferry, Massey, Sklyarenko} for more on the stakes, virtues, and limitations of the various classical approaches to homology along these lines.}
Homology theories constructed on such principles are said to be of \emph{\v{C}ech type} and tend, in contradistinction to singular homology, to be highly responsive to spaces' shapes.
\item Within the aforementioned trio of works, Milnor's most celebrated contribution is probably his introduction, in \cite{Milnor}, of the derived functors $\mathrm{lim}^s$ of the inverse limit to the subject.
These are a fundamental focus of the present work, and we review their definition and computation below.
More immediately significant for our purposes is the heuristic that the functors $\mathrm{lim}^s$ $(s\geq 1)$ measure the failure of the inverse limit functor to be exact.
\end{itemize}
Continuing with the second point, \v{C}ech homology's failure to satisfy the Exactness Axiom might be viewed as stemming from its negligence of the functors $\mathrm{lim}^s$ $(s\geq 1)$; in this perspective, attention to those functors should hold the key to reconciling \v{C}ech-type approaches to homology with the Eilenberg-Steenrod axioms.

The most systematic elaboration of this last insight is \emph{strong homology}.
Developed chiefly by Sibe Marde\v{s}i\'{c} in the decade of the 1980s and in the reference work \cite{SHINA}, this homology theory both subsumes a number of classical constructions and admits more contemporary formulations, in terms of homotopy limits \cite{Cordier} or $(\infty,1)$-categories \cite[\S 7.1.6]{Lurie}, for example. It exhibits the following desirable features:
\begin{itemize}
    \item It satisfies the Eilenberg-Steenrod axioms for paracompact pairs $(X,A)$ and, hence, coincides with singular homology on the category $\mathsf{HCW}$ of spaces having the homotopy type of a CW complex.
    \item It coincides with Steenrod homology on the category $\mathsf{MC}$ of metric compacta.
    \item It is strong shape invariant.
\end{itemize}
It was against this background that Marde\v{s}i\'{c} and Prasolov turned, in 1988, to the question of whether strong homology is additive \cite{SHINA}.
They considered as well the closely related question of \emph{having compact supports} (a property which seems first to have appeared in \cite[Exercise IX.C.1]{EileStee}).

\begin{defn}
A homology theory is \emph{compactly supported on the class $\mathcal{C}$ of topological spaces} if for every natural number $p$ and $X\in \mathcal{C}$, the inclusion maps $i_{K K'}: K\to K'$ and $K\to X$ among compact subsets $K, K'$ of $X$ induce an isomorphism
$$\mathrm{colim}(H_p(K),i^{*}_{K K'})\to H_p(X).$$
\end{defn}
At heart, theirs were questions about the compatibilities between the main prospective continuity properties of homology theories on any robust class of topological spaces.

Marde\v{s}i\'{c} and Prasolov succeeded in showing that if the Continuum Hypothesis is assumed then strong homology fails to be either additive or compactly supported, even on so restricted a class as that of the closed subsets of $\mathbb{R}^2$.
Since that time, the obvious complementary questions have remained open, namely, those of whether it is consistent with the ZFC axioms for strong homology to be additive, or to have compact supports, on any natural class of topological spaces properly extending $\mathsf{MC}$.
These appear as Questions 7.3 and 7.4 of \cite{SVHDL}, for example (see also that work's introduction for a fuller history of this problem). Our main theorem provides an answer to both these questions.

\begin{thm} \label{main_result}
If the ZFC axioms are consistent with the existence of a weakly compact cardinal, then the ZFC axioms
are consistent with the assertion that
strong homology is compactly supported, and therefore additive, on the class
of locally compact separable metric spaces.
\end{thm}

\emph{Weakly compact cardinals} were first considered by Tarski in his study of the compactness properties of infinitary languages.
Their existence is a strong form of the Axiom of Infinity. Weakly compact cardinals are \emph{strongly inaccessible}; in particular, their existence implies the consistency of the ZFC axioms and consequently cannot be deduced from those axioms alone.
On the other hand, the assumption of their existence is a modest one in comparison to most large cardinal principles, and the conventional wisdom has long
been that the existence of a weakly compact cardinal is indeed consistent with ZFC (even if this statement is, in any meaningful sense, unprovable); see \cite[\S1.4 and p. 471]{higher_infinite}.

In the process of arguing Theorem \ref{main_result}, we will in fact show that on the class of locally compact separable metric spaces (or, equivalently, locally compact Polish spaces), the two aforementioned questions of additivity and compact supports are one and the same; see Theorem \ref{Om_add_cpt_supt} below.
Our main theorem, in consequence, is in some sense sharp: Lisica in \cite{Lisica} exhibited two separable metric spaces on which strong homology fails in ZFC to have compact supports, but neither of those examples is locally compact. A restriction to metric spaces, similarly, is dictated by Prasolov's non-metrizable ZFC counterexample to the additivity of strong homology \cite{Prasolov}.
Still, it is unclear if Theorem \ref{main_result} can be extended to the class of $\sigma$-compact Polish spaces
or to the class of locally compact metric spaces (in the latter case one would expect to replace
\emph{weak compactness} with \emph{supercompactness}
as a hypothesis in order to manage summands of unrestricted cardinality).

Our witness to Theorem \ref{main_result} is the model $V^{\Pbb}$ of \cite{SVHDL}, in which $\Pbb$ is a finitely supported iteration of \emph{Hechler forcing} which is of weakly compact length.
Our proof has four main components. In the first, we introduce algebraic objects, \emph{$\Om$-systems of abelian groups}, which capture the computational core of the questions discussed above. In the second, we show that the arguments of \cite{SVHDL} generalize from their original target of a single inverse system $\mathbf{A}$ to apply to all $\Om$-systems of abelian groups.
In our third step, we conclude from these arguments that, in $V^{\Pbb}$, the higher derived limit functors $\limn{s}$ are additive in the context of $\Om$-systems of abelian groups.
In our fourth and concluding step, we show that this form of additivity is equivalent both to strong homology being additive, and to strong homology having compact supports, on the category of locally compact separable metric spaces.

Before beginning, we note one further appealing aspect of our results. In any model of set theory witnessing the conclusion of Theorem \ref{main_result}, two of the main homology theories extending Steenrod homology beyond the class of metric compacta coincide on the category of locally compact separable metric spaces; each theory thereby inherits the known desirable features of the other.
These two theories are strong homology and Steenrod-Sitnikov homology.
The latter is simply defined as the compactly-supported extension of Steenrod homology; it was introduced largely to witness the duality theorem we cite just below \cite{Milnor2,Sitnikov}.
We collect these effects in the following corollary; here and below, we restrict our attention to integral homology groups.
\begin{cor} If the ZFC axioms are consistent with the existence of a weakly compact cardinal, then the ZFC axioms
are consistent with the following:
\begin{itemize}
\item Steenrod-Sitnikov homology is strong shape invariant on the category of locally compact separable metric spaces.
\item Strong homology satisfies Sitnikov duality for any locally compact subspace of a sphere (see \cite[p. 80]{Milnor2}).
\item On the category of locally compact separable metric spaces, the Steenrod-Sitnikov and strong homology theories coincide, and are axiomatized by the axioms for Steenrod homology together with the axiom of compact supports.
\end{itemize}
\end{cor}
The corollary is immediate from Theorem \ref{main_result} together with the definitions and assertions in the first pages, e.g., of \cite{SHINA,Milnor2}.
It may be read as describing a highly canonical homology theory in any model witnessing Theorem \ref{main_result}. See also \cite{Petkova} and \cite{Melikhov} and the references therein; the latter justly emphasizes the computational benefits of uniqueness results like this corollary.

\section{Background and preliminary definitions}
\label{background:sec}

The reader is referred to \cite{set_theory:Kunen} for background on set theory in general
and forcing in particular;
large cardinals are treated extensively in \cite{higher_infinite}.
See \cite{SSH} for a complete introduction to strong homology and its associated homological algebra; a more
concise account may be found in \cite{SHINA}.

We adopt von Neumann's convention that an ordinal is the set of its predecessors.
In particular $0= \emptyset$ is the smallest ordinal, $n+1 = \{0,\ldots,n\}$, and $\omega$ is the set of finite ordinals --- a set which coincides with the nonnegative integers.
All counting and indexing will begin at 0 unless otherwise indicated.
We will typically use $\alpha$, $\beta$, and $\gamma$ to denote ordinals.
Cardinals are ordinals which are strictly larger in cardinality than their predecessors.
These are generally denoted $\kappa$, $\lambda$ and $\nu$.

If $X$ is a set, we will write $[X]^n$ for the set of all $n$-element subsets of $X$.
If $X$ is linearly ordered, we will identify elements of $[X]^n$ with their increasing
enumerations, viewed as elements of $X^n$.
We construe $X^0$ as the singleton consisting of the null sequence.
For $0\leq i<n$ and $\vec{x}\in X^n$, we will write $x_i$ for the $i\Th$ coordinate of $\vec{x}$ and write $\vec{x}^i$ for the element of $X^{n-1}$ obtained from $\vec{x}$ by deleting $x_i$.

Recall that a partial order is a \emph{lattice} if each pair of elements $x,y$ has a least upper upper bound
$x \join y$ and greatest lower bound $x \meet y$.

\begin{notn} If $x,y \in\,^ \omega\omega$, we will write $x \leq y$
to mean $x(i) \leq y(i)$ for all $i \in \omega$.
Henceforth we will write $\Om$ for the lattice
$({^\omega}\omega,\leq)$.
\end{notn}

We turn now to our fundamental objects of interest.
First among these are inverse systems of abelian groups, together with the main categorical contexts for such objects.

\begin{defn}
Let $\Gamma$ be a partial order.
A \emph{$\Gamma$-indexed inverse system of abelian groups} $\mathbf{G}=(G_x,p_{x,y},\Gamma)$ consists of abelian groups $G_x$ for all $x\in\Gamma$ and bonding homomorphisms $p_{x,y}:G_y\to G_x$ for each $x\leq y$ in $\Gamma$; these additionally must satisfy the equations $p_{x,x}=\mathrm{id}$ and $p_{x,y} p_{y,z}=p_{x,z}$ for each $x\leq y\leq z$ in $\Gamma$.
A \emph{level morphism} $\mathbf{g}:(G_x,p_{x,y},\Gamma)\to(H_x,q_{x,y},\Gamma)$ is a collection of morphisms $g_x:G_x\to H_x$ $(x\in\Gamma)$ satisfying $q_{x,y} g_y = g_x p_{x,y}$ for all $x\leq y$ in $\Gamma$.
These systems of abelian groups and level morphisms together form the category which we denote $\mathsf{Ab}^{\Gamma}$.
\end{defn}
For the following, see \cite[\S 11.1]{SSH}.
\begin{prop}
$\mathsf{Ab}^{\Gamma}$ is an abelian category for any partial order $\Gamma$, and a sequence $$\mathbf{G}\xrightarrow{\mathbf{g}}\mathbf{H}\xrightarrow{\mathbf{h}}\mathbf{K}$$
is exact at $\mathbf{H}$ if and only if it is levelwise exact, i.e., if and only if $\mathrm{ker}(h_x)=\mathrm{im}(g_x)$ for all $x\in\Gamma$.
\end{prop}
The bulk of our work below may be viewed as taking place in the category $\mathsf{Ab}^\Omega$.
Additivity questions, however, involve interactions between inverse systems with varying index-sets; these are best formulated in the category pro-$\mathsf{Ab}$ of pro-abelian groups.

\begin{defn}
Inverse systems of abelian groups, or \emph{pro-abelian groups} for short, comprise the objects of pro-$\mathsf{Ab}$.
The morphisms therein between two such systems $(G_x,p_{x,y},\Gamma)$ and $(H_x,q_{x,y},\Lambda)$ are the elements of the set
$$\mathrm{lim}_{y\in\Lambda}\,\mathrm{colim}_{x\in\Gamma}\,\mathrm{Hom}(G_x,H_y).$$
\end{defn}
Here we opted for a more abstract definition since the morphisms of pro-$\mathsf{Ab}$ will not form our primary concern; for more concrete (and verbose) descriptions of the morphisms in pro-$\mathsf{Ab}$, readers are referred to \S 1.1 of \cite{SSH} or \cite{Shape1}.
The essential heuristic is that morphisms and, hence, systems which are cofinal in one another are equivalent in pro-$\mathsf{Ab}$.

Appearing just above, of course, is the well-known inverse limit functor $\mathrm{lim}$, which we regard as a functor $\mathsf{Ab}^{\Lambda}\to\mathsf{Ab}$ for some fixed $\Lambda$; it will simplify discussion hereafter to additionally assume that $\Lambda$ is a lattice.
We adopt the following definitions.
\begin{defn} For any $\mathbf{G}=(G_x,p_{x,y},\Lambda)$ in $\mathsf{Ab}^\Lambda$,
$$\mathrm{lim}\,\mathbf{G}:=\{(g_x)\in\prod_{x\in\Lambda} G_x\mid p_{x,y}(g_y)=g_x\textnormal{ for all }x\leq y\textnormal{ in }\Lambda\},$$ 
and for any $\mathbf{g}:\mathbf{G}\to\mathbf{H}$ in $\mathsf{Ab}^{\Lambda}$, $$\mathrm{lim}\,\mathbf{g}:=\prod_{x\in\Lambda} g_x:\mathrm{lim}\,\mathbf{G}\to\mathrm{lim}\,\mathbf{H}.$$ 
\end{defn}
This functor is \emph{left exact}, meaning that $\mathrm{lim}$ applied to a short exact sequence
\begin{align}
\label{shortexactsequence}
\mathbf{0}\to\mathbf{G}\to\mathbf{H}\to\mathbf{K}\to\mathbf{0}
\end{align}
will, in general, only preserve exactness at the $\mathrm{lim}\,\mathbf{G}$ and $\mathrm{lim}\,\mathbf{H}$ terms.
When coupled with its \emph{derived functors} $\mathrm{lim}^s$, on the other hand, $\mathrm{lim}$ does preserve the exactness of (\ref{shortexactsequence}), within a long exact sequence of the following form:
$$0\to\mathrm{lim}\,\mathbf{G}\to\mathrm{lim}\,\mathbf{H}\to\mathrm{lim}\,\mathbf{K}\to\mathrm{lim}^1\,\mathbf{G}\to\mathrm{lim}^1\,\mathbf{H}\to\mathrm{lim}^1\,\mathbf{K}\to\mathrm{lim}^2\,\mathbf{G}\to\cdots$$
The functors $\mathrm{lim}^s$ admit complete characterizations along these lines --- or, simply, as the derived functors of $\mathrm{lim}$; readers are referred to \cite[Book III]{SSH} for a comprehensive introduction to this material.
Therein the following more concrete characterization appears as well \cite[Corollary 11.47]{SSH}; we precede it with a few further notational conventions.
\begin{notn}
Suppose that $\Gbf = (G_x, p_{x,y},\Lambda)$ is an inverse system of abelian groups where
$\Lambda$ is a lattice.
If $s>0$ and $\vec{x}$ is an $s$-tuple of elements of a lattice,
we will write $\meet \vec{x}$ for $\bigwedge_{i=0}^{s-1} x_i$. Moreover, if $\vec{x} \in \Lambda^s$, we will let $G_{\vec{x}}$ denote $G_{\meet \vec{x}}$; similarly, $p_{\vec{x},\vec{y}}$ will be used to denote
$p_{\meet \vec{x},\meet\vec{y}}$.
We will write $\Lambda^{[s]}$ for the subcollection of $\Lambda^s$ consisting of $\leq$-increasing $s$-tuples.
Lastly, we will adopt $\mathrm{lim}^0$ as an alternate notation for $\mathrm{lim}$.
\end{notn}
\begin{prop}
\label{Proposition:classical_lims} $\mathrm{lim}^s\,\Gbf\cong H^s(\mathcal{C}(\Gbf))$ for any $\Gbf\in\mathsf{Ab}^\Lambda$ and $s\geq 0$, where the cochain complex $\mathcal{C}(\Gbf)$ is defined as follows.
If $q \geq 0$, let $C^q(\Gbf)=\prod_{\vec{x} \in \Lambda^{[q+1]}} G_{\vec{x}}$.
Let $C^q (\Gbf) = 0$ otherwise.
For each $q \geq 0$ define
$\delta: C^q(\Gbf)  \rightarrow C^{q+1}(\Gbf)$
by
$$
\delta \Phi(\vec{x}):=
\sum_{i=0}^{q+1}(-1)^i p_{\vec{x},\vec{x}^i}(\Phi(\vec{x}^i)).
$$
In fact this construction determines a natural isomorphism of the functors $\mathrm{lim}^s$ and $H^s(\mathcal{C}(\,\cdot\,))$ from $\mathsf{Ab}^\Lambda$ to $\Lambda$.
\end{prop}
For our purposes, however --- especially in Section \ref{Section:Hechlers} --- a cochain complex less constrained by the order-relations among elements of $\Lambda$ will be far more useful.
To that end, we record the following alternate characterization of $\mathrm{lim}^s$; this will be our working definition of the functor in all that follows.
\begin{notn} If $\sigma$ is a permutation, we let $\sgn(\sigma) = 1$ if $\sigma$ is even and $-1$ if $\sigma$ is odd.  
We will write $\sigma \vec{x}$
to denote the result of permuting the coordinates of
$\vec{x}$ with $\sigma$ so that the $i\Th$ coordinate of
$\sigma \vec{x}$ is $x_{\sigma(i)}$.
In addition, we will write $\Gbf\!\restriction\Gamma$ for the restriction of $\Gbf$ to a suborder $\Gamma$ of $\Lambda$.
\end{notn}
\begin{defn}
An element $\Phi \in \prod_{\vec{x} \in \Lambda^s} G_{\vec{x}}$ is \emph{alternating} if
whenever $\sigma$ is a permutation of $s$ and $\vec{x} \in \Lambda^s$,
$\Phi_{\sigma \vec{x}} = \sgn(\sigma) \Phi_{\vec{x}}$.
If $s \geq 0$, let $C^s_{alt}(\Gbf)$ denote the set of all elements of $\prod_{\vec{x} \in \Lambda^{s+1}} G_{\vec{x}}$
which are alternating.
Let $C_{alt}^s (\Gbf) = 0$ if $q < 0$.
For each $s \geq 0$, formally define
$\delta: C^s_{alt}(\Gbf) \rightarrow C^{s+1}_{alt}(\Gbf)$ exactly as in Proposition \ref{Proposition:classical_lims}, so that this $\delta$ differs from that one only in ranging over a wider collection of indices $\vec{x}$.
\end{defn}
\begin{thm}
$\mathrm{lim}^s\,\Gbf\cong H^s(\mathcal{C}_{alt}(\Gbf))$ for any $\Gbf\in\mathsf{Ab}^\Lambda$ and $s\geq 0$; in fact, $\mathrm{lim}^s$ and $H^s(\mathcal{C}_{alt}(\,\cdot\,))$ are naturally isomorphic functors $\mathsf{Ab}^\Lambda\to\mathsf{Ab}$.
\end{thm}
For completeness, we sketch the proof of the first assertion, from which the second follows fairly directly.
\begin{defn} As above, let $\Lambda$ be a lattice and let $\tau$ denote the topology on $\Lambda$ generated by the sets $U_x:=\{y\in\Lambda\mid y\leq x\}$, where $x$ ranges through $\Lambda$. For any inverse system $\Gbf$ indexed by $\Lambda$, let $\mathcal{F}_{\Gbf}$ denote the sheaf on $(\Lambda,\tau)$ defined by $U\mapsto \mathrm{lim}\,\Gbf\!\restriction\! U$, together with the natural restriction maps $\mathcal{F}_{\Gbf}(U)\to\mathcal{F}_{\Gbf}(V)$ for any $V\subseteq U$ in $\tau$.
\end{defn}
Underlying the following lemma is the fact that the above operation induces an equivalence of categories between that of the inverse systems of abelian groups indexed by $\Lambda$ and that of the sheaves of abelian groups on $(\Lambda,\tau)$; $\check{H}^s(X;\mathcal{F})$ denotes the $s\Th$ \emph{\v{C}ech cohomology group} of a topological space $X$ with coefficients in the sheaf $\mathcal{F}$.
\begin{lem}[\hspace{.001 cm}\cite{JENSEN}] \label{jensen_lemma} $\mathrm{lim}^s\,\Gbf\cong\check{H}^s((\Lambda,\tau);\mathcal{F}_{\Gbf})$ for all $s\geq 0$.
\end{lem}
Our characterization of $\mathrm{lim}^s\,\Gbf$ follows immediately from Lemma \ref{jensen_lemma} together with the following observations:
\begin{itemize}
\item $\{U_x\mid x\in\Lambda\}$ refines all open covers of $(\Lambda,\tau)$.
\item $\mathrm{lim}\,\Gbf\!\restriction\! U_x\,=G_x$ for all $x\in\Lambda$.
\item \v{C}ech cohomology groups may be computed using alternating cochains \cite[I.20 Proposition 2]{SERRE}.
\end{itemize}
See \cite[Section 2.2]{SVHDL} for more detailed argumentation of this characterization in the particular case of the inverse system $\mathbf{A}$; as noted, this inverse system forms the template for those of the next section, and indeed, for those of the remainder of the paper.

One last feature of derived limits is worth recording before proceeding; a variant appears as Corollary 12.15 in \cite{SSH}.
\begin{prop}
\label{Proposition:productive_limits}
For any $s\geq 0$ and collection $\{\Gbf_i\mid i\in I\}$ of pro-abelian groups, the groups $\prod_{i\in I}\mathrm{lim}^s\,\Gbf_i$ and $\mathrm{lim}^s\,\prod_{i\in I}\Gbf_i$ are isomorphic.
In particular, since in abelian categories finite products and coproducts coincide, $\mathrm{lim}^s$ is finitely additive.
\end{prop}

We conclude this section by recalling a well-known characterization of \emph{weakly compact cardinals}; see Theorem 7.8 of \cite[Ch. 2]{higher_infinite}.

\begin{thm} \label{wc_char} 
An uncountable cardinal $\kappa$ is weakly compact if whenever $n \in \omega$, $\nu < \kappa$, and
$f:[\kappa]^n \to \nu$ there is a set $H \subseteq \kappa$ with $|H| = \kappa$ such that
$f \restriction [H]^n$ is constant. 
\end{thm}

This is the only property of weakly compact cardinals which we will need below, and we will invoke it solely by way of Lemma \ref{uniformizing_lemma}, which follows from Theorem \ref{wc_char} together with arguments recorded in \cite{SVHDL}.

\section{$\Om$-systems of abelian groups}

In this section we will introduce some specialized terminology and symbols which will render the results of later
sections easier to state. 
The main notion introduced here is that of an $\Om$-system of abelian groups; these appear in calculations of the strong homology groups of countable disjoint unions of compact metric spaces. 

An \emph{$\Om$-system of abelian groups $\Gcal$} is specified by
an indexed collection $\{G_{n,k} \mid n,k \in \omega\}$
of finitely generated abelian groups along with compatible homomorphisms
$p_{n,j,k}: G_{n,k} \to G_{n,j}$ for
each $n$ and $j \leq k$.
Such data gives rise to the following additional objects:
\begin{itemize}

\item 
For each $x \in \Om$ define 
${\displaystyle G_x := \bigoplus_{n=0}^\infty G_{n,x(n)}}$
and
${\displaystyle \overline{G}_x := \prod_{n=0}^\infty G_{n,x(n)}}$.
We regard $G_x$ as a subset of $\overline{G}_x$.

\item
For each $x \leq y \in \Om$ a homomorphism
$p_{x,y}:\overline{G}_y \to \overline{G}_x$
defined by $p_{x,y} := \prod_{n=0}^\infty p_{n,y(n),x(n)}$.

\end{itemize}
Observe that $p_{x,y}$ maps $G_y$ into $G_x$.
We will denote the inverse system $(G_x,p_{x,y},\Om)$
by $\Gbf$ and $(\overline{G}_x,p_{x,y},\Om)$ by $\overline{\Gbf}$. We write $\overline{\Gbf}/\Gbf$ for the inverse system $(\overline{G}_x/G_x,p^*_{x,y},\Om)$ in which $p^*_{x,y}$ is the mapping of quotient groups induced by $p_{x,y}$.

Finally, define $\Gbf_k$ to be the inverse system indexed by $\Om$ where
$(G_k)_x$ consists of the elements of $G_x$ with support contained
in the $k\Th$ coordinate.  A crucial point in what follows is the fact that $\Gbf\cong\bigoplus_{k\in\omega}\Gbf_k$ in the category of pro-abelian groups.

We will see in Section \ref{limnadd_cpt_supt_section} that the question of the $\limn{s}$-additivity of $\Om$-systems is deeply related to that of the additivity of strong homology. 
For now, to provide intuition, we merely state a main context in which $\Om$-systems arise. 
Given compact metric spaces $\{X_n\mid n\in\omega\}$, we have a resolution of each $X_n$ as a sequence of compact polyhedra $X_{n,k}$. 
For any given $s$, we obtain an $\Om$-system of abelian groups $\Gcal$ by taking the homology pro-groups $H_s$ of these resolutions. 
There is also a standard related way of constructing a resolution of $\coprod_{n \in \omega} X_n$ (see e.g. \cite[Theorem 6]{SHINA}). For $s$ as above, the pro-group $\Gbf$ is the $s\Th$ homology pro-group of this resolution. 

A classical obstruction to the additivity of strong homology is the nonvanishing of the higher derived limits of the pro-abelian group $\Abf$ given by the $\Omega$-system $\Acal$ in which
$A_{n,k} = \Z^k$ and $p_{n,j,k}:\Z^{k} \to \Z^j$ is the projection onto the first $j$ coordinates ($\overline{\Abf}$ is denoted $\Bbf$ in, e.g., \cite{SVHDL}, \cite{SHINA}).
This is illustrated in the following theorem concerning countable topological sums of the $d$-dimensional Hawaiian earring $Y^{(d)}$; for concreteness, readers may first fix a basepoint $*$ of $S^d$ and then regard $Y^{(d)}$ as the subspace of $\prod_{i=0}^\infty S^d$ consisting of those elements for which at most one coordinate is not $*$.

\begin{thm}[\hspace{.001 cm}\cite{SHINA}]
For $d \geq 1$ and $p$,
$$
\overline{H}_p(Y^{(d)}) \approx
\begin{cases}
0, & p \ne 0,d \\
\prod_{i=0}^\infty \Zbb, & p = d \\
\Zbb, & p = 0
\end{cases}
\qquad  \textrm{and} \qquad
\overline{H}_p(\coprod_{i=0}^\infty  Y^{(d)}) \approx
\begin{cases}
0, & p  > d \\
\lim \Abf, & p = d \\
\limn{d-p} \Abf, & p < d \textrm{ and } p \ne 0 \\
\limn{d} \Abf \oplus \big( \bigoplus_{i \in \omega} \Zbb \big), & p = 0
\end{cases}
$$
where $\overline{H}_p(\cdot)$ denotes the $p\Th$ strong homology group.
\end{thm}

In particular, strong homology is additive for a countable sum of $d$-dimensional Hawaiian earrings if and only if
$\limn{s} \Abf = 0$ for $0 < s \leq d$.

\section{Higher coherence and the derived limit $\limn{s}$}

The study of the triviality of coherent families of functions indexed by $^\omega\omega$ dates to \cite{SHINA}; there, Marde\v{s}i\'{c} and Prasolov showed that the additivity of strong homology implies that every coherent family is trivial.
Higher-dimensional variants of these notions were introduced in \cite{SHDLST}; these encoded the behaviors of the higher derived limits associated to the system $\mathcal{A}$ described above.
In this section, we further generalize these notions to $s$-coherence and $s$-triviality for arbitrary $\Om$-systems $\Gcal$, and we analyze their relationship to these systems' associated higher derived limits.

Fix an $\Om$-system of abelian groups $\Gcal$ and an $s \geq 1$.
\begin{defn}
$\Phi\in C_{alt}^{s-1}(\overline{\Gbf})$ is \emph{an $s$-coherent family for $\Gcal$} if
$\delta \Phi$ is an element of $C_{alt}^s(\Gbf) \subseteq C_{alt}^s(\overline{\Gbf})$.
$\Phi$ is \emph{an $s$-trivial family for $\Gcal$}
if there is a $\Psi\in C_{alt}^{s-1}(\overline{\Gbf})$ such that $\delta \Psi = 0$
and $\Psi - \Phi \in C_{alt}^{s-1}(\Gbf)$.
\end{defn}
If $\Gcal$ is clear from context, we will simply write ``$\Phi$ is $s$-coherent'' or ``$\Phi$ is $s$-trivial.''
Observe that $\Phi$ is $s$-coherent if and ony if it represents an element of 
$\limn{s-1} \overline{\Gbf}/\Gbf$.
Similarly $\Phi$ is $s$-trivial if and only if it represents an element of the range of the map
$\limn{s-1} \overline{\Gbf} \to \limn{s-1} \overline{\Gbf}/\Gbf$.

We will often consider the subfamilies $\Phi\restriction X$ of such $\Phi$ associated to some subset $X$ of $\Omega$.
By Lemmas \ref{tail-cof_criterion} and \ref{triv_from_restriction} below, when $X$ is $\leq^*$-cofinal, conclusions about these subfamilies tend to apply to $\Phi$ itself.

\begin{remark}
If $\Gcal$ is the system $\Acal$ mentioned before,
then the above definition of $s$-trivial is equivalent to the one given in \cite[Lemma 2.5]{SVHDL}.
The two definitions agree when $\limn{1} \overline{\Gbf} = 0$ but in our general
setting the present definition seems more appropriate.
To see that such limits may indeed fail to vanish, consider the $\Om$-system $\mathcal{Z}$ in which $Z_{n,k} = \Zbb$
and $p_{n,k,k+1}(z) = 2z$; in this case, $\limn{1} \overline{\mathbf{Z}}=\prod_\omega \Z_2/\Z \ne 0$, where $\Z_2$ denotes the $2$-adic integers.
Such systems arise in the computation of the strong homology groups of sums of dyadic solenoids.
\end{remark}

Though the behavior of $\limn{1}\overline{\Gbf}$ may vary with our choice of $\Om$-system $\Gcal$, the following lemma shows that its higher derived limits will not.

\begin{lem} \label{limn_barG_vanish}
Let $\Gcal$ be an $\Om$-system of abelian groups.
If $s \geq 2$, then $\limn{s} \overline{\Gbf} = 0$.
\end{lem}

\begin{proof}
Let $\Gcal$ and $s \geq 2$ be fixed.
Observe that $\overline{\Gbf} = \prod_{k=0}^\infty \Gbf_k$.
By \cite[Theorems 11.52 and 14.9]{SSH}, $\limn{n} \Gbf_k = 0$ for all $k$.
By Proposition \ref{Proposition:productive_limits}, 
$$\limn{s} \overline{\Gbf} = \limn{s} \prod_{k=0}^\infty \Gbf_k \cong
\prod_{k=0}^\infty \limn{s} \Gbf_k = 0.$$
\end{proof}

For each $n \in \omega$, define $x \leq^n y$ to mean that
$x(i) \leq y(i)$ whenever $n \leq i$.
If $x \leq^n y$ for some $n$, then we write $x \leq^* y$.

\begin{lem} \label{tail-cof_criterion}
If $X \subseteq \Om$ is $\leq^*$-cofinal, then $X$ is $\leq^n$-cofinal for some $n$.
\end{lem}

\begin{proof}
Suppose not.
For each $n$, let $y_n \in \Om$ be such that there is no $x \in X$ with
$y_n \leq^n x$.
Define $y \in \Om$ by $y(n) = \max \{y(i) \mid i \leq n\}$, noting
that $y_n \leq^n y$ for all $n$.
By assumption, there is an $x \in X$ and $n$ such that
$y \leq^n x$.
But now $y_n \leq^n y \leq^n x$, a contradiction.
\end{proof}

The next lemma generalizes \cite[Lemma 2.7]{SVHDL} and will play a similar role in our generalization of that work's results. 

\begin{lem} \label{triv_from_restriction}
Let $\Gcal$ be an $\Om$-system of abelian groups.
If $X \subseteq \Om$ is $\leq^*$-cofinal,
$\Phi$ is an $s$-coherent family for $\Gcal$, and
$\Phi \restriction X$ is $s$-trivial, then
$\Phi$ is $s$-trivial.
\end{lem}

\begin{proof}
By Lemma \ref{tail-cof_criterion},
$X$ is $\leq^n$-cofinal for some $n$.
It follows from the finite additivity of $\mathrm{lim}^{s-1}$ that there exist $\Psi \in C^{s-1}_{alt} (\bigoplus_{i =0}^{n-1} \Gbf_i )$ and $\Theta \in C^{s-1}_{alt}(\bigoplus_{i =n}^\infty \Gbf_i)$
such that $\Phi = \Psi + \Theta +\delta\Upsilon$ for some $\Upsilon\in C^{s-2}_{alt}(\Gbf)$ (with $\delta\Upsilon=0$ if $s=1$).
Clearly $\Psi$ and $\delta\Upsilon$ are $s$-trivial, and since $\Phi \restriction X = \Psi \restriction X + \Theta \restriction X+ \delta\Upsilon \restriction X$, 
the family $\Theta \restriction X$ is $s$-trivial; hence by \cite[Theorem 14.9]{SSH}, $\Theta$ is $s$-trivial as well.
This concludes the proof.
\end{proof}

The next lemma records the relationship between the triviality of $s$-coherent families for $\Gcal$ and the vanishing of the derived limits of $\Gbf$.
It may be read as asserting that $\limn{s+1}\Gbf=0$ if and only if every $(s+1)$-coherent family for $\Gcal$ is $(s+1)$-trivial. 

\begin{lem} \label{limn_equiv}
Let $\Gcal$ be an $\Om$-system of abelian groups and
$s \geq 1$. $\limn{s+1} \Gbf = 0$ if and only if the natural
map $\limn{s} \overline{\Gbf} \to \limn{s} \overline{\Gbf}/\Gbf$ is a surjection.
\end{lem}

\begin{proof}
By Lemma \ref{limn_barG_vanish},
$\limn{s+1} \overline{\Gbf} = 0$.
The conclusion now follows from the exactness of the portion
$$
\limn{s} \overline{\Gbf} \to
\limn{s} \overline{\Gbf}/\Gbf \to \limn{s+1} \Gbf \to 0
$$
of the long exact sequence associated to $\overline{\Gbf}/\Gbf$.
\end{proof}

\section{A vanishing theorem for $\limn{s} \Gbf$}
\label{Section:Hechlers}

In this section, we prove a generalization of the main result in \cite{SVHDL}.
We recall the following definitions;
see \cite[VII,VIII]{set_theory:Kunen} for undefined terminology and a complete introduction to 
forcing and finite-support iterated forcing.

\begin{defn} 
\emph{Hechler forcing} consists of pairs $q=(t_q,f_q)$ where
$f_q \in\, ^\omega\omega$ and $t_q$ is a finite initial part of $f_q$.
We will refer to $t_q$ as the \emph{stem} of $q$.
Define $r \leq q$ if $t_q$ is an initial part of $t_r$ and $f_q \leq f_r$.
\end{defn}

\begin{defn}
If $\kappa$ is an ordinal, $\Hbb_\kappa$ denotes the \emph{finite-support iteration}
of Hechler forcing of length $\kappa$.
Elements of $\Hbb_\kappa$ are finite functions $q$ with domain contained
in $\kappa$ such that if $\alpha \in \dom(q)$ then $q(\alpha)$ is a \emph{nice $\Hbb_\alpha$-name
for a Hechler condition}.
The order on $\Hbb_\kappa$ is defined by $q \leq p$ if $\dom(p) \subseteq \dom(q)$ and
for all $\alpha < \kappa$,
$p \restriction \alpha \leq q \restriction \alpha$ in $\Hbb_\alpha$ and
if $\kappa = \alpha + 1$ and $\alpha \in \dom(p)$,
$q \restriction \alpha$ forces $q(\alpha) \leq p(\alpha)$.
\end{defn}

\begin{notn}
For $\alpha < \kappa$, we let $\dot h_\alpha$ denote
the $\alpha\Th$ Hechler real.
In general we will follow the \emph{dot convention} whereby $\dot x$ denotes a name for the element
$x$ which it represents in a generic extension.
\end{notn}

The next theorem is the central techical result of our paper.
The instance $\Gcal = \Acal$ is the main result of \cite{SVHDL};
our proof will consist in showing that, with minor modifications, the arguments of \cite{SVHDL} generalize to \emph{any} $\Om$-system of abelian groups $\Gcal$.

\begin{thm} \label{coherent->trivial}
If $\kappa$ is weakly compact then $\Hbb_\kappa$ forces that for any $s\geq 1$ and $\Om$-system $\Gcal$,
every $s$-coherent family for $\Gcal$ is trivial.
\end{thm}

When combined with Lemma \ref{limn_equiv}, Theorem \ref{coherent->trivial} yields the following corollary.

\begin{cor} \label{vanishing}
If $\kappa$ is a weakly compact cardinal,
then $\Hbb_\kappa$ forces that whenever
$\Gcal$ is an $\Om$-system of abelian groups and
$s \geq 2$, then $\limn{s} \Gbf = 0$.
\end{cor}

\begin{proof}[Proof of Theorem \ref{coherent->trivial}.]
Before beginning, we will recall the combinatorial mainspring of \cite{SVHDL}, namely, a higher-dimensional $\Delta$-system lemma packaging some of the large cardinal effects of $\kappa$ on indexed families of conditions in $\Hbb_\kappa$.
We then treat the case of $s=1$. The cases of $s> 1$ involve increasingly elaborate combinations of ``error-terms''; to describe these effectively will require some further preparatory remarks, which will heavily reference \cite{SVHDL}.
Portions of our argument will even follow \cite{SVHDL} verbatim.

The argument-strategy in all cases is the following: we will show that for any $q\in\Hbb_\kappa$ forcing that some family $\dot{\Phi}$ for some $\Om$-system $\dot{\Gcal}$ is $s$-coherent, there exists a $q_{\varnothing}\leq q$ and a $\Hbb_\kappa$-name $\dot{B}$ such that
$$q_{\varnothing}\Vdash\text{``}\dot{B}\text{ is a cofinal subset of }\kappa\text{ and }\dot{\Phi}\restriction \{\dot{h}_\alpha\mid\alpha\in \dot{B}\} \text{ is $s$-trivial''}.$$
As any $\Hbb_\kappa$-condition forcing that $\dot{B}$ is cofinal in $\kappa$ forces also that $\{\dot{h}_\alpha\mid\alpha\in\dot{B}\}$ is $\leq^*$-cofinal in $\Om$, this will imply $\dot{\Phi}$ to be $s$-trivial by Lemma 4 above.
For ease of reading, we will tend to abbreviate the indices $h_\alpha$ organizing our argument as $\alpha$, writing $\Phi_{\alpha,\beta}$ in place of $\Phi_{(h_\alpha,h_\beta)}$ and $p_{\alpha\wedge\beta,\beta}$ for the bonding map $\overline{G}_{h_\beta}\to\overline{G}_{h_\alpha\wedge h_\beta}$, for example.
Relatedly, we will write $\Phi\restriction B$ for families $\Phi\restriction \{h_\alpha\mid\alpha\in B\}$ as above.
Within such a framework, the utility of the following lemma should be easy to imagine.
The assumption that $\kappa$ is weakly compact is indispensible to this lemma, and this, in fact, is the only direct use which we will make of this assumption.
\begin{lem}[\hspace{.001 cm}\cite{SVHDL}] \label{uniformizing_lemma}
  Let $n$ be a positive integer and let $\{q_{\vec{\alpha}}\,|\,\vec{\alpha} \in [\kappa]^n\}$ be a family of
  conditions in $\Hbb_\kappa$.
  Let $u_{\vec{\alpha}} = \mathrm{dom}(q_{\vec{\alpha}})$.
  Then there is an unbounded set $A \subseteq \kappa$, a family
  $\langle u_{\vec{\alpha}} \mid \vec{\alpha} \in [A]^{<n}
  \rangle$, a natural number $\ell$, and a set of stems $\langle t_i \mid i < \ell \rangle$
  such that
  \begin{enumerate}
    \item $|u_{\vec{\alpha}}| =
    \ell$ for all $\vec{\alpha} \in [A]^n$, and if $\eta$ is the $i^{\mathrm{th}}$ element of
    $u_{\vec{\alpha}}$ then $t_{q_{\vec{\alpha}}(\eta)} = t_i$.
    \item $A$ and $\langle u_{\vec{\alpha}} \mid \vec{\alpha} \in [A]^{\leq n} \rangle$
    satisfy \begin{enumerate}
    \item for all $\vec{\alpha} \in [A]^{<n}$,
    \begin{enumerate}
      \item if $\beta \in A$ and $\vec{\alpha} < \beta$, then
      $u_{\vec{\alpha}} < \beta$,
      \item if $\vec{\beta} \in [A]^{\leq n}$ satisfies
      $\vec{\alpha} \sqsubseteq \vec{\beta}$, then $u_{\vec{\alpha}} \sqsubseteq
      u_{\vec{\beta}}$,
      \item the set $\{u_{\vec{\alpha}
      ^\frown \langle \beta \rangle} \mid \beta \in A\backslash(\max(\vec{\alpha})+1)\}$ forms a $\Delta$-system with root $u_{\vec{\alpha}}$;
    \end{enumerate}
    \item for all $m \leq n$ and all $\vec{\alpha}, \vec{\beta} \in [A]^m$,
    \begin{enumerate}
    \item $|u_{\vec{\alpha}}| = |u_{\vec{\beta}}|$, and
    \item if $\vec{\alpha}$ and $\vec{\beta}$ are aligned, then $u_{\vec{\alpha}}$
    and $u_{\vec{\beta}}$ are aligned.
    (Two finite sets of ordinals $u$ and $v$ are \emph{aligned} if $|u|=|v|$ and $|u\cap\alpha|=|v\cap\alpha|$ for all $\alpha\in u\cap v$.) 
    \end{enumerate}
    \end{enumerate}
    \item $q_{\vec{\beta}} \restriction u_{\vec{\alpha}}
    = q_{\vec{\gamma}} \restriction u_{\vec{\alpha}}$ for all $\vec{\alpha} \in [A]^{<n}$ and $\vec{\beta}, \vec{\gamma} \in
    [A]^n$ such that $\vec{\alpha} \sqsubseteq \vec{\beta}$ and $\vec{\alpha}
    \sqsubseteq \vec{\gamma}$.
  \end{enumerate}
\end{lem}
\begin{proof}
This is a combination of Lemmas 3.3 and 4.3 in \cite{SVHDL}.
\end{proof}

We now argue the case of $s=1$.
Let $q\in\Hbb_\kappa$ force that $\dot{\Phi}$ is $1$-coherent for $\dot{\Gcal}$.
By the above remarks, we may, for simplicity, take $\dot{\Phi}$ to be a sequence of $\Hbb_\kappa$-names $\{\dot{\Phi}_\alpha\mid \alpha\in\kappa\}$ with each $\dot{\Phi}_\alpha$ in $\dot{\overline{G}}_{\dot{h}_\alpha}$.
Let $A_0=\kappa\backslash(\max(\dom(q))+1)$; much as in the proof of \cite[Theorem 4.5]{SVHDL}, we begin by fixing for each $\alpha<\beta$ in $A_0$ a $q_{\alpha\beta}\leq q$ and $n_{\alpha\beta}\in\omega$ for each $\alpha<\beta$ in $A_0$ such that 
\begin{align*}q_{\alpha\beta}\Vdash\text{ ``}\dot{h}_\alpha\leq\dot{h}_\beta\text{ and }\dot{\Phi}_\alpha(n)=\dot{p}_{\alpha\beta}(\dot{\Phi}_\beta)(n)\text{ for all }n>n_{\alpha\beta}\text{''}.\end{align*}
This is possible because we may first extend $q$ to
$q \cup \{\langle \beta, (\emptyset, \dot h_\alpha)\rangle\}$ --- which forces
$\dot h_\alpha \leq \dot h_\beta$ --- and then further extend to the desired condition
using the fact that $q$ forces that $\dot{\Phi}$ is $1$-coherent.
Applying the weak compactness of $\kappa$, we thin $A_0$ to a cofinal $A_1\subseteq A_0$ such that $n_{\alpha\beta}$ equals some fixed $n$ for all $\alpha<\beta$ in $A_1$.
We then apply Lemma \ref{uniformizing_lemma} to the collection $\langle q_{\alpha\beta}\mid (\alpha,\beta)\in [A_1]^2\rangle$ to find a cofinal $A\subseteq A_1$, $u_\emptyset$, $\langle u_{\langle \alpha \rangle} \mid
  \alpha \in A \rangle$, $\ell$, and $\langle t_i \mid i < \ell \rangle$ as
  in the statement of the lemma.
  
  Let $q_{\varnothing}=q_{\alpha,\beta}\restriction u_{\varnothing}$ for some $\alpha<\beta$ in $A$.
  Similarly, for each $\alpha\in A$ choose some $\beta>\alpha$ in $A$ and let $q_\alpha=q_{\alpha,\beta}\restriction u_\alpha$.
  By Lemma \ref{uniformizing_lemma}, these definitions are independent of our choices.
  Observe also that as $q_{\varnothing}=\bigcap_{(\alpha,\beta)\in [A]^2}q_{\alpha,\beta}$ and each $q_{\alpha,\beta}\leq q$, we have that $q_{\varnothing}\leq q$.
  Let $\dot{B}$
  be a $\Hbb$-name for the set of $\alpha \in A$ such that $q_\alpha \in \dot{G}$,
  where $\dot{G}$ is the canonical $\Hbb_\kappa$-name for the $\Hbb_\kappa$-generic filter.
  This is the $q_{\varnothing}$ and $\dot{B}$ we seek, as described above.
\begin{claim} \label{unbounded_b_claim}
    $q_{\varnothing} \Vdash\text{``}\dot{B} \text{ is cofinal in } \kappa\text{''}$.
  \end{claim}
  \begin{proof} Fix $r\leq q_\varnothing$; it will suffice to find an $\alpha\in (A\backslash\eta)$ such that $r$ and $q_\alpha$ are compatible.
  To this end, note that as $\langle u_{\langle\alpha\rangle}\mid\alpha\in A\rangle$ forms a $\Delta$-system, there exists an $\alpha\in (A\backslash\eta)$ with $u_{\langle\alpha\rangle}\backslash u_\varnothing\cap \dom(r)=\varnothing$. Since $q_\alpha\restriction u_\varnothing=q_\varnothing$ and $q_\varnothing\geq r$, the conditions $q_\alpha$ and $r$ are indeed compatible, as desired.
  \end{proof}
\begin{claim} \label{linking_claim_1}
    $q_{\varnothing}$ forces that, for all $\alpha < \beta$ in $\dot{B}$, there is a
    $\gamma \in (A \setminus \beta + 1)$ such that $q_{\alpha, \gamma}$
    and $q_{\beta, \gamma}$ are in $\dot{G}$.
\end{claim}  
   \begin{proof}
    Fix $r \leq q_\varnothing$ and $\alpha < \beta$ such that $r$ forces both $\alpha$
    and $\beta$ to be in $\dot{B}$.
    By the definition of $\dot{B}$, we may
    assume that $r$ extends both $q_\alpha$ and $q_\beta$.
    It then suffices
    to find a $\gamma \in (A \setminus \beta + 1)$ such that $r, q_{\alpha, \gamma}$,
    and $q_{\beta, \gamma}$ all have a common extension $r'$.
    By construction, the
    families $\left\{ u_{\langle \alpha, \gamma \rangle} \setminus
    u_{\langle \alpha \rangle} ~ \middle| ~ \gamma \in (A \setminus \beta + 1) \right\}$
    and $\left\{ u_{\langle \beta, \gamma \rangle} \setminus u_{\langle \beta \rangle}
    ~ \middle| ~ \gamma \in (A \setminus \beta + 1) \right\}$ each consist of pairwise
    disjoint sets.
    We can therefore find a $\gamma \in (A \setminus \beta + 1)$
    such that $(u_{\langle \alpha, \gamma \rangle} \setminus
    u_{\langle \alpha \rangle})\, \cap\, \dom(r) = \emptyset$ and
    $(u_{\langle \beta, \gamma \rangle} \setminus
    u_{\langle \beta \rangle}) \,\cap\, \dom(r) = \emptyset$.
    By item (2.b.ii) of Lemma \ref{uniformizing_lemma}, the domains $u_{\langle \alpha, \gamma \rangle}$
    and $u_{\langle \beta, \gamma \rangle}$ are aligned, hence by item (1) of that lemma, the stems in $q_{\alpha,\gamma}$ and $q_{\beta,\gamma}$ match whenever their domains intersect.
    Any lower bound (see \cite[Lemma 4.1]{SVHDL}) for $r$ and
    $q_{\alpha, \gamma}$ and $ q_{\beta, \gamma }$ is then an $r'$ such as we had sought.
  \end{proof}
Fix now a $\Hbb_\kappa$-generic filter with $q_\varnothing\in G$.
We work in $V[G]$ and recall the convention that, e.g., $x$ denotes
the interpretation of $\dot x$ in the generic extension. 
We claim that the family $\Psi$ defined for each $\alpha\in B$ by
\[ \Psi_\alpha(k)=
\begin{cases} 
      0 & k\leq n \\
      \Phi_\alpha(k) & k>n 
   \end{cases}
\]
$1$-trivializes $\Phi\restriction B$.
This reduces to showing that $\delta\Psi=0$.
Observe, though, that $\delta\Psi_{\alpha,\beta}(k)=0$ for all $\alpha<\beta$ in $B$ and $k\leq n$.
For $\alpha<\beta$ in $B$ and $k>n$, fix a $\gamma$ as given by Claim \ref{linking_claim_1}; that claim together with the definitions of $q_{\alpha,\gamma}$ and $q_{\beta,\gamma}$ implies the following:
\begin{align*}
\delta\Psi_{\alpha,\beta}(k) & = p_{\alpha\wedge\beta,\beta}(\Psi_\beta)(k)-p_{\alpha\wedge\beta,\alpha}(\Psi_\alpha)(k) \\ & = p_{\alpha\wedge\beta,\beta}\circ p_{\beta,\gamma}(\Psi_\gamma)(k)-p_{\alpha\wedge\beta,\alpha}\circ p_{\alpha,\gamma}(\Psi_\gamma)(k) \\ & = p_{\alpha\wedge\beta,\gamma}(\Psi_\gamma)(k) - p_{\alpha\wedge\beta,\gamma}(\Psi_\gamma)(k) \\ & = 0\,.
\end{align*}
This concludes the argument of the case $s=1$.

One consequence of the above argument is a lemma which we will apply in the following section; we pause to record it:
\begin{lem} \label{wc_n=2again}
The following is forced by $\Hbb_\kappa$:
for every $c:[\kappa]^2 \to \omega$ there is a
cofinal $B \subseteq \kappa$ and $n \in \omega$ such that for every
$\alpha < \beta \in B$, there is a $\gamma > \alpha,\beta$
such that $c(\alpha,\gamma) = c(\alpha,\beta) = n$
and $h_\alpha,h_\beta \leq h_\gamma$.
\end{lem}
To see this, let $q\in\Hbb_\kappa$ force that $\dot{c}$ is a function $[\kappa]^2\to\omega$ and proceed as above: choose for each $\alpha<\beta$ in $A_0=\kappa\backslash(\max(\dom(p))+1)$ a $q_{\alpha,\beta}\leq q$ and $n_{\alpha\beta}$ such that
\begin{align*}q_{\alpha\beta}\Vdash\text{ ``}\dot{h}_\alpha\leq\dot{h}_\beta\text{ and }\dot{c}(\alpha,\beta)=n_{\alpha\beta}\text{''}.\end{align*}
The argument through to Claim \ref{linking_claim_1} then applies unchanged, and proves the lemma.

For the cases of $s>1$, our argument-strategy will be the following: let $q$ force that $\dot{\Phi}\in \dot{C}_{\mathrm{alt}}^{s-1}(\dot{\overline{\Gbf}})$ is $s$-coherent for $\dot{\Gcal}$, or in other words that $\delta\dot{\Phi}\in \dot{C}_{\mathrm{alt}}^{s}(\dot{\Gbf})$.
We will define a $q_\varnothing\leq q$ and $\mathbb{P}_\kappa$-names $\dot{B}$ and $\dot{\Theta}$ such that
\begin{align}\label{strategy}
q_\varnothing\Vdash\text{``}\dot{B}\text{ is a cofinal subset of }\kappa\text{ and }\dot{\Theta}\in \dot{C}_{\mathrm{alt}}^{s-1}(\dot{\Gbf})\text{ and }\delta(\dot{\Theta}\restriction\dot{B})=\delta(\dot{\Phi}\restriction\dot{B})\text{''}.
\end{align}
Clearly, $q_\varnothing$ then forces that $\delta((\dot{\Phi}-\dot{\Theta})\restriction \dot{B})=0$, and hence that $\dot{\Psi}\restriction \dot{B}:=(\dot{\Phi}-\dot{\Theta})\restriction \dot{B}$ is an $s$-trivialization of $\dot{\Phi}$, as desired.

The terms of the family $\dot{\Theta}\restriction \dot{B}$ are defined from differences among the terms of the family $\dot{\Phi}$; as noted, these definitions grow increasingly elaborate as the parameter $s$ rises.
We will follow the scheme introduced in \cite{SVHDL} for describing them; recalling it will occupy the next two pages.

\begin{defn}\label{61}
  For any nonempty $\tau$ in $[\kappa]^{<\omega}$, a \emph{subset-initial segment
  of $\tau$} is a sequence $\sigma_1 \subseteq \cdots \subseteq\sigma_m \subseteq \tau$
  such that
  \begin{itemize}
    \item $m \leq |\tau|$ and
    \item $|\sigma_i| = i$ for all $i$ with $1 \leq i \leq m$.
  \end{itemize}
  We write $\vec{\sigma} \vartriangleleft \tau$ to indicate that
  $\vec{\sigma}$ is a subset-initial segment of $\tau$.
  If $m = |\tau|$, then
  we call $\vec{\sigma} = \langle \sigma_1, \cdots, \sigma_m \rangle\vartriangleleft\tau$ a \emph{long string} or
  \emph{long string for $\tau$}.
\end{defn}

Fix now an $\Om$-system $\Gcal$ and suppose that $\langle h_\alpha \mid \alpha
< \kappa \rangle$ is an injective sequence of elements of ${^\omega}\omega$.
Suppose that for each positive integer $s$ the family $\Phi^s=\{
\Phi_{\vec{\alpha}}\mid \vec{\alpha} \in [\kappa]^s\}$ is $s$-coherent for $\Gcal$, where each such $\Phi_{\vec{\alpha}}$ is an element of $G_f$, where $f=\wedge_{i < s}\, h_{\alpha_i}$.
Let
$\vec{\Phi}$ denote the family $\{ \Phi_{\vec{\alpha}} \mid
\vec{\alpha} \in [\kappa]^{<\omega}, ~ \vec{\alpha} \neq \varnothing \}$.
Suppose also that to each nonempty $\tau \in [\kappa]^{<\omega}$ we have assigned
an ordinal $\alpha_\tau < \kappa$ in such a way that
\begin{itemize}
  \item if $\tau = \{\eta\}$, then $\alpha_\tau = \eta$ and
  \item if $\rho \subsetneq \tau$, then $\alpha_\rho < \alpha_\tau$.
\end{itemize}
Given a nonempty $\tau \in [\kappa]^{<\omega}$ and subset-initial segment
$\vec{\sigma} \vartriangleleft \tau$, write $\vec{\alpha}[\vec{\sigma}]$
to denote the sequence $\langle \alpha_{\sigma_i} \mid 1 \leq i \leq
|\vec{\sigma}| \rangle$; note that this sequence is increasing by assumption.
We will also sometimes write $\vec{\alpha}(\tau)$ to denote the increasing sequence enumerating $\tau$.

For $\vec{\alpha} \in [\kappa]^{<\omega}$ of length at least two, let
\[
  e^{\vec{\Phi}}(\vec{\alpha}) = \sum_{i < |\vec{\alpha}|} (-1)^i
  \Phi_{\vec{\alpha}^i}
\]
When the family $\vec{\Phi}$ is clear from context, it will
be omitted from the superscript above. Observe also that, for obvious reasons, we have notationally suppressed the bonding maps $p_{\wedge_{j\in s}\,\alpha_j,\wedge_{j\in s\backslash\{i\}}\,\alpha_j}$ rendering the above sum meaningful; we will continue to do so below.

Since each family $\Phi^s$ is $s$-coherent, the function $e(\vec{\alpha})$ is finitely supported for
any $\vec{\alpha} \in [\kappa]^{<\omega}$ of length at least two. We write $\mathtt{e}(\vec{\alpha})$ for the restriction of $e(\vec{\alpha})$ to its support.

We now record a series of formal equalities. First, for all $\vec{\alpha}
\in [\kappa]^{<\omega}$ of length at least three, let
\[
  \mathsf{d} e(\vec{\alpha}) = \sum_{i<|\vec{\alpha}|}(-1)^i e(\vec{\alpha}^i).
\]
Observe that $\mathsf{d} e(\vec{\alpha})$ is well-defined and equals $0$ for all such $\vec{\alpha}$: since $(\vec{\alpha}^i)^j = (\vec{\alpha}^j)^{i-1}$ for $j < i$,
when each $e(\vec{\alpha}^i)$ is expanded, terms appear in pairs of opposite-signed but otherwise identical expressions, and consequently all cancel.
 
If $L$ is some linear combination of the form
\[
  \sum_{i<\ell} a_i\, e(\vec{\alpha}_i)
\]
with all $\vec{\alpha}_i$ of length at least three then let
\[
  \mathsf{d} L=\sum_{i<\ell} a_i\, \mathsf{d} e(\vec{\alpha}_i)
\]
For such an $L$ also let
\[
  L*\beta=\sum_{i<\ell} a_i\, e(\vec{\alpha}_i,\beta)
\]
for any $\beta\in \kappa$ with $\vec{\alpha}_i <
\beta$ for all $i < \ell$, where $e(\vec{\alpha}_i,\beta) = e(\vec{\alpha}_i
{^\frown} \langle\beta\rangle)$. It is to ensure that operations such as $\mathsf{d}$ and $*$ are defined that we begin from a family of families $\vec{\Phi}$; ultimately, however, only the data of $\Phi^s$ will matter to our stage-$s$ applications of this framework (see \cite[p. 25]{SVHDL} for further discussion of this point).
Finally, if $j$ is less than $|\vec{\alpha}_i|$ for all $i$ then let
\[
  L^j=\sum_{i<\ell} a_i\, e(\vec{\alpha}^j_i)
\]

For integers $s\geq 2$ we now recursively define interrelated
\begin{itemize}
\item expressions $\mathcal{A}_s(\rho)$ parametrized by $\rho \in [\kappa]^s$,
\item expressions $\mathcal{S}_s(\tau)$ and
$\mathcal{C}_s(\tau)$ parametrized by $\tau \in [\kappa]^{s+1}$, and
\item statements $\mathfrak{u}_s(\tau)$ parametrized by $\tau \in [\kappa]^{s+1}$.
\end{itemize}
These expressions will depend at once on $\Gcal$, $\vec{\Phi}$, and the
collection $\{\alpha_\tau \mid \tau \in [\kappa]^{<\omega}\}$ fixed above, but this dependence will always be contextually clear enough that we will notationally ignore it. 

To begin, let
\[
  \mathcal{A}_2(\rho) = e(\vec{\alpha}(\rho), \alpha_\rho)
\]
for each $\rho \in [\kappa]^2$.
Next, given $s$ with $2 \leq s < \omega$ and $\tau \in [\kappa]^{s+1}$, if
$\mathcal{A}_s(\tau^i)$ has been defined for all $i \leq s$, let
\begin{align*}
  \mathcal{S}_s(\tau) &= \mathsf{d} e(\vec{\alpha}(\tau), \alpha_\tau) -
  \sum_{i < n + 1} (-1)^i \mathsf{d}[\mathcal{A}_s(\tau^i) * \alpha_\tau], \\
  \mathcal{C}_s(\tau) &= e(\vec{\alpha}(\tau)) - \sum_{i < s + 1}
  (-1)^i \mathcal{A}_s(\tau^i),
\end{align*}
and let $\mathfrak{u}_s(\tau)$ denote the conjunction of the following two
statements:
\begin{itemize}
  \item There exists an $\mathtt{e}$ such that $\mathtt{e}(\vec{\alpha}
  [\vec{\sigma}]) = \mathtt{e}$ for every long string $\vec{\sigma}
  \vartriangleleft \tau$.
  \item For all nonempty $\rho, \sigma$ with $\rho \subsetneq \sigma
  \subseteq \tau$, we have $h_{\alpha_\rho} \leq h_{\alpha_\sigma}$.
\end{itemize}
Lastly, let
\[
  \mathcal{A}_{s+1}(\tau) = (-1)^{s+1}\mathcal{C}_s(\tau) * \alpha_\tau.
\]
Several comments are at this point in order:
\begin{enumerate}
\item The interrelations of these expressions is strictly a matter of the arithmetic of the operations $e$, $\mathsf{d}$, $*$, and $+$, and the order-relations of the functions associated to $\vec{\alpha}$.
In other words, though these interrelations were first argued for the $\Om$-system $\mathcal{A}$, they will hold for any $\Om$-system $\Gcal$ taking its place.
Below, we record the main lemma connecting these expressions, then proceed directly to the $s>1$ case of our theorem; the proof of the lemma in \cite[pp. 25--28]{SVHDL} applies in our context without change.
\item The roles of the above-described expressions in our argument are as follows: $\mathcal{A}_s(\rho)$ is the value we assign to the functions $\Theta_{\vec{\alpha}(\rho)}$ evoked at (\ref{strategy}) above.
The assertion $\mathfrak{u}_s(\tau)$ records those order-relations and difference-relations which any $q_{\vec{\alpha}(\rho)}$ in the $\Hbb_\kappa$-generic filter $G$ enforces.
In the presence of $\mathfrak{u}_s(\tau)$, the expression $\mathcal{S}_s(\tau)$ amounts to the computation showing that $\mathcal{C}_s(\tau)=0$, and in light of our assignments $\mathcal{A}_s(\rho)=\Theta_{\vec{\alpha}(\rho)}$, the global vanishing of $\mathcal{C}_s\restriction [B]^{s+1}$ amounts to the relation $\delta(\Phi\restriction B)=\delta(\Theta\restriction B)$ we had taken as our goal at (\ref{strategy}).
\end{enumerate}
\begin{lem} \label{s_c_lemma}
$\mathcal{S}_s(\tau)=0$ for all $2 \leq s < \omega$ and $\tau \in [\kappa]^{s+1}$.
If also
  $\mathfrak{u}_s(\tau)$ holds, then
  \[
    0=\mathcal{S}_s(\tau)=(-1)^{s+1}\mathcal{C}_s(\tau).
  \]
\end{lem}
The lemma combines Lemmas 6.3 and 6.4 of \cite{SVHDL}.

Fix now $s>1$.
We will show in the manner described above that every $s$-coherent family for an $\Om$-system $\Gcal$ in $V^{\Hbb_\kappa}$ is $s$-trivial.
To that end, let $q\in\Hbb_\kappa$ force that $\dot{\Phi}^s=\{\dot{\Phi}_{\vec{\alpha}}\mid\vec{\alpha}\in [\kappa]^{s}\}$ is $s$-coherent for an $\Om$-system $\dot{\Gcal}$.
Observe that we may without loss of generality (i.e., up to $\Om$-system isomorphism) assume that the generating set of each $\dot{G}_{n,k}$ comprising $\dot{\Gcal}$ lies in $V$.
Let $A_0 = \kappa \setminus (\max(\dom(q)) + 1)$. Much as before, for all $\vec{\alpha}\in[A_0]^{s+1}$
  there exists a condition $q_{\vec{\alpha}} \leq p_0$ such that
  $q_{\vec{\alpha}} \Vdash\text{``}\dot{h}_{\alpha_0} \leq \dots
  \leq \dot{h}_{\alpha_i} \leq \dots \leq \dot{h}_{\alpha_{s}}\text{''}$ and
  $q_{\vec{\alpha}}$ decides the value of
  $\dot{\mathtt{e}}(\vec{\alpha})$ to be equal to some $e_{\vec{\alpha}} \in V$; this is because $\dot{\mathtt{e}}(\vec{\alpha})$ is forced to have finite support, with its values falling thereon in the finitely-generated groups $\dot{G}_{n,k}$.
  Apply the weak compactness of $\kappa$ to thin $A_0$ out to a cofinal $A_1\subseteq\kappa$ such that
  $e_{\vec{\alpha}}$ equals some fixed $e$ for all
  $\vec{\alpha} \in [A_1]^{s+1}$.

Now apply Lemma \ref{uniformizing_lemma} to $\left\langle q_{\vec{\alpha}} ~ \middle| ~ \vec{\alpha} \in [A_1]^{s+1} \right\rangle$ to find an unbounded $A \subseteq A_0$
  together with sets $\langle u_{\vec{\alpha}} \mid \vec{\alpha} \in [A]^{\leq s}
  \rangle$, a natural number $\ell$, and stems $\langle t_i \mid i < \ell \rangle$
  as in the statement of the lemma.
  Next, define conditions $\left\langle q_{\vec{\alpha}} ~ \middle| ~ \vec{\alpha}
  \in [A]^{< s} \right\rangle$
  as follows: for each $\vec{\alpha}$ in $[A]^{\leq s}$ let $\vec{\beta}$ be
  an element of $[A]^{n+1}$ such that $\vec{\alpha} \sqsubseteq \vec{\beta}$
  and let $q_{\vec{\alpha}} = q_{\vec{\beta}} \restriction u_{\vec{\alpha}}$.
  As above, by Lemma \ref{uniformizing_lemma}, these definitions
  are independent of all of our choices of $(n+1)$-tuples $\vec{\beta}$.
  Moreover, the fact that $q_\varnothing=\bigcap_{\vec{\alpha}\in [A]^{s+1}}q_{\vec{\alpha}}$ implies that $q_\varnothing\leq q$.
  
 We claim that $q_\varnothing$ forces $\dot{\Phi}^s$ to be trivial.
 This we argue by first partitioning $A$ into $s + 1$ disjoint and unbounded subsets $\{\Gamma_i\mid 1\leq i\leq s+1\}$.
 Let $\dot{B}$ be a $\mathbb{P}_\kappa$-name for the set of
  $\alpha \in \Gamma_1$ such that $q_{\langle\alpha\rangle} \in \dot{G}$, where $\dot{G}$ is the canonical name for the $\mathbb{P}_\kappa$-generic filter.
  By exactly the same reasoning as in the proof of Claim \ref{unbounded_b_claim},
  \[
    q_\varnothing \Vdash\text{``}\dot{B} \text{ is unbounded in } \kappa\text{''}.
  \]
\begin{claim}\label{strings}
   Fix $\tau\in [\kappa]^{s+1}$.
   The condition $q_\varnothing$ forces the following to hold
   in $V^{\mathbb{P}_\kappa}:$

   Suppose that $m$ and $\left\{\alpha_\sigma ~ \middle| ~ \sigma\in[\tau]^{<m}\textnormal{
   and }\sigma\neq\emptyset\right\}$ are such that
   \begin{itemize}
     \item $1< m\leq s+1$,
     \item $\alpha_{\{\eta\}} = \eta$ for all $\eta \in \tau$,
     \item $\alpha_\rho<\alpha_\sigma$ whenever $\rho$ is a proper subset of $\sigma$,
     \item $\alpha_\sigma\in \Gamma_{|\sigma|}$ for all nonempty
     $\sigma\in[\tau]^{<m}$, and
     \item for any $1\leq\ell< m$ and subset-initial segment $\vec{\sigma}
     \vartriangleleft \tau$ of length $\ell$, we have $q_{\vec{\alpha}[\vec{\sigma}]}
     \in \dot{G}$.
     In particular, $\eta \in \dot{B}$ for all $\eta \in \tau$.
   \end{itemize}
   Then there exist $\left\{\alpha_\sigma ~ \middle| ~ \sigma\in[\tau]^m\right\}\subseteq\Gamma_m$ which
   together with $\left\{\alpha_\sigma ~ \middle| ~ \sigma\in([\tau]^{<m}\backslash\{\emptyset\})\right\}$
   satisfy
   \begin{itemize}
     \item $\alpha_\rho<\alpha_\sigma$ whenever $\rho$ is a proper subset of $\sigma$, and
     \item for any subset-initial segment $\vec{\sigma} \vartriangleleft \tau$
     of length $m$, we have $q_{\vec{\alpha}[\vec{\sigma}]} \in \dot{G}$.
   \end{itemize}
 \end{claim}
 The proof of this claim (Claim 6.7 in \cite{SVHDL}) applies wholesale in our context as well.
 We therefore proceed to the conclusion of our argument.
 As in the $s=1$ case, we now work in $V[G]$, where $G$ is a $\Hbb_\kappa$-generic filter containing $q_\varnothing$; undotted names will again denote their realizations in $V[G]$. We begin by specifying for each nonempty $\sigma \in
 [B]^{\leq s+1}$ an ordinal $\alpha_\sigma < \kappa$ in such a way
 that the collection $\left\{\alpha_\sigma ~ \middle| ~ \sigma \in [B]^{\leq s+1}\right\}$ satisfies
 the following:
 \begin{itemize}
   \item $\alpha_{\{\eta\}} = \eta$ for all $\eta \in B$,
   \item $\alpha_\rho < \alpha_\sigma$ whenever $\rho$ is a proper subset of
   $\sigma$,
   \item $\alpha_\sigma \in \Gamma_{|\sigma|}$ for all nonempty $\sigma \in
   [B]^{\leq s + 1}$, and
   \item for every $\tau \in [B]^{s+1}$ and every subset-initial segment
   $\vec{\sigma} \vartriangleleft \tau$, we have $q_{\vec{\alpha}[\vec{\sigma}]}
   \in G$.
 \end{itemize}

 The construction of $\left\{\alpha_\sigma ~ \middle| ~ \sigma \in [B]^{\leq s+1}\right\}$
 proceeds via a straightforward recursion on $|\sigma|$, invoking
 Claim \ref{strings}. The $s$-coherent family $\Phi^s$ and collection $\left\{\alpha_\sigma
 ~ \middle| ~ \sigma \in [B]^{\leq s + 1} \backslash \{\emptyset\}\right\}$ determine
 expressions $\mathcal{A}_s(\rho)$ for all $\rho \in [B]^s$ and
 expressions $\mathcal{S}_s(\tau)$ and $\mathcal{C}_s(\tau)$ and statements
 $\mathfrak{u}_s(\tau)$ for all $\tau \in [B]^{s+1}$.  Note that our choices of
  $q_{\vec{\alpha}[\vec{\sigma}]}$, $B$, and $\left\{\alpha_\sigma ~ \middle| ~ \sigma \in [B]^{\leq s + 1}\right\}$
 ensure that $\mathfrak{u}_s(\tau)$ holds for all $\tau \in [B]^{s+1}$.

 For each $\rho \in [B]^s$, let $\Theta_{\vec{\alpha}(\rho)} = \mathcal{A}_s(\rho)$.
 It follows from the definition of $\mathcal{A}_s(\rho)$ and the fact that $\mathfrak{u}_s(\tau)$
 holds for all $\tau \in [B]^{s+1}$ that $\mathcal{A}_s(\rho)$
 is an element of $G_{\wedge\vec{\alpha}(\rho)} \subseteq \overline{G}_{\wedge\vec{\alpha}(\rho)} $.
 To show that $\Phi_s$ is $s$-trivial, it will suffice to show that
 \begin{align}\label{last}
  e(\vec{\alpha}(\tau)) = \sum_{i < s+1} (-1)^i \Theta_{\vec{\alpha}(\tau^i)}
 \end{align}
 for all
 $\tau \in [B]^{s+1}$, thereby establishing $\delta( \Phi \restriction \Bcal) = \delta (\Theta \restriction \Bcal)$ and hence
 realizing the plan described at (\ref{strategy}) above.
 Fix such a $\tau$.
 The above equation is easily seen to be equivalent to the assertion that $\mathcal{C}_s(\tau) = 0$.
 By Lemma \ref{s_c_lemma} and the fact that $\mathfrak{u}_s(\tau)$ holds, $0=\mathcal{S}_s(\tau) = (-1)^{s+1}\mathcal{C}_s(\tau)$.
 In consequence, $\mathcal{C}_s(\tau) = 0$.
 This shows that equation (\ref{last}) holds for an arbitrary $\tau$ hence that $\Phi^s$ is trivial.
\end{proof}

\section{An additivity lemma for $\limn{s} \Gbf$}

The additivity of $\limn{0}$ is a ZFC fact (this follows, for example, from the proof of \cite[Theorem 9]{SHINA}).
We are now in a position to show that $\Hbb_\kappa$ forces the additivity of $\limn{s}$ for $\Om$-systems, for all $s\geq 0$. 

\begin{lem} \label{add_lem}
If $\kappa$ is a weakly compact cardinal, then
$\Hbb_\kappa$ forces
that whenever $\Gcal$ is an
$\Om$-system of abelian groups and $s \geq 0$,
$\bigoplus_k \limn{s} \Gbf_k \to \limn{s} \bigoplus_k \Gbf_k$ is an isomorphism.
\end{lem}

\begin{proof}
If $s \geq 2$, then by Theorem \ref{vanishing},
$\limn{s} \Gbf = \limn{s} \Gbf_k = 0$ and there is
nothing to show.
The case of $s=0$ was discussed above.

Suppose therefore that $s =1$ and let $\Phi \in C_{alt}^1(\Gbf)$ represent an element of $\limn{1} \Gbf$.
Define
$$c(\alpha,\beta) = \min \{k \in \omega \mid \Phi (h_\alpha,h_\beta) \in \bigoplus_{i =0}^k G_{(h_\alpha \meet h_\beta) (i)}\}.$$
By Lemma \ref{wc_n=2again}, there is a cofinal $B \subseteq \kappa$ such that
if $\alpha < \beta$ are in $B$, then there is a $\gamma > \alpha,\beta$ such
that $c(\alpha,\gamma) = c(\beta,\gamma) = k$.
By using Lemma \ref{tail-cof_criterion} and replacing $B$ with a cofinal subset if necessary, we may assume that
$X:=\{h_\beta \mid \beta \in B\}$ is $\leq^n$-cofinal for some $n \geq k$.
Since
$$\delta \Phi (h_{\alpha},h_{\beta},h_{\gamma}) = \Phi(h_{\beta},h_{\gamma}) - \Phi (h_\alpha,h_\gamma) +
\Phi (h_\alpha,h_\beta)
= 0$$
it follows that $\Phi(h_\alpha,h_\beta) \in \bigoplus_{i=0}^n G_{(h_\alpha \meet h_\beta) (i)}$. 
Applying Proposition \ref{Proposition:productive_limits}, let $\Phi=\Psi+\Theta+\delta\Pi$ where $\Psi \in C_{alt}^1(\bigoplus_{i=0}^n \Gbf_i)$, $\Theta \in C_{alt}^1(\bigoplus_{i = n+1}^\infty \Gbf_i)$, and $\Pi\in C_{alt}^0(\bigoplus_{i=0}^\infty\Gbf_i)$.
Since $X$ is $\leq^n$-cofinal and $\Theta \restriction X = -\delta\Pi\restriction X$, 
\cite[Theorem 14.9]{SSH} implies that $\Theta$ represents $0$ in {$\lim^1 \bigoplus_{i = n+1}^\infty \Gbf_i$}. 
Thus $\Phi$ and $\Psi$ represent the same element of $\lim^1 \Gbf$.
Since $\limn{s}$ is finitely additive,
$\Psi$ --- and hence $\Phi$ --- represent elements in the range of 
$\bigoplus_k \limn{s} \Gbf_k \to \limn{s} \bigoplus_k \Gbf_k$; this concludes the proof.

\end{proof}

\begin{remark}
We do not know if every $1$-coherent family being $1$-trivial implies the $s=1$ case of Lemma \ref{add_lem}, as happens for $s>1$. 
\end{remark}

\section{The relationship between additivity and compact supports for strong homology}\label{limnadd_cpt_supt_section}

In this section, we relate the compact supports and additivity properties of strong homology with $\limn{s}$-additivity for $\Om$-systems.
Recall that the strong homology functors are denoted $\overline{H}_p$; their precise definition is sufficiently involved that we refer readers to \cite{SSH} or \cite{SHINA} for its full description.
Our main theorem is the following:

\begin{thm} \label{Om_add_cpt_supt}
The following are equivalent: 
\begin{enumerate}
    \item \label{cpt_supt}
    Strong homology has compact supports on the class of locally compact separable metric spaces.
    \item \label{cpt_add}
    Whenever $\{X_i\mid i\in\omega\}$ are compact metric, the natural map $\bigoplus_i(\overline{H}_p(X_i))\to\overline{H}_p(\coprod_iX_i)$ is an isomorphism for each $p$.
    \item \label{Om_limn_add}
Whenever $\Gcal$ is an $\Om$-system of abelian groups and $s \geq 0$, the natural map
$\bigoplus_k \limn{s} \Gbf_k \to \limn{s} \bigoplus_k \Gbf_k$ is an isomorphism.
\end{enumerate}
\end{thm}

By Lemma \ref{add_lem}, if $\kappa$ is weakly compact then item (\ref{Om_limn_add}) is forced by $\Hbb_\kappa$, hence the implication (\ref{Om_limn_add})$\Rightarrow$(\ref{cpt_supt}) will complete the proof of Theorem \ref{main_result}. 

Proving Theorem \ref{Om_add_cpt_supt} will require more of the machinery of strong homology than we have used so far; in particular, we will invoke both the first and second Miminoshvili exact sequences. 
For the reader's convenience, we begin by recalling these sequences, along with two related corollaries instrumental in computations of strong homology groups.
For explicit constructions of the $s$-stage strong homology groups $H_n^{(s)}$ therein, as well as proofs of Theorem \ref{M_seqs} and its corollaries, the reader is referred to \cite[Section 4]{SHINA}. 
In the following, $H_n(\Xbf)$ denotes the pro-abelian group consisting of the $n\Th$ singular homology groups of spaces in $\Xbf$, with connecting homomorphisms those induced by the maps in $\Xbf$. 

\begin{thm} \label{M_seqs}
For all integers $p$ and $s\geq 0$ there exist functors $H_p^{(s)}\colon$pro-$\mathsf{Top}\to\mathsf{Ab}$ and maps $j_p^{s-1,s}\colon H_p^{(s)}\to H_p^{(s-1)}$ which together satisfy:
\begin{enumerate}
    \item[(i)] 
    $H_n^{(0)}(\Xbf)=\lim H_n(\Xbf)$.
    \item[(ii)] \label{First_seq}
    For each integer $n$ and pro-topological space $\mathbf{X}$, there exists an exact sequence
    
    \begin{center}
        \begin{tikzcd}[row sep=tiny, column sep=small]
        0\ar[r]&
        \limn{1} H_{n+1}(\mathbf{X})\ar[r]&
        \overline{H}_{n}^{(1)}(\mathbf{X})\ar[r]&
        \overline{H}_{n}^{(0)}(\mathbf{X})\ar[r]&
        \limn{2} H_{n+1}(\mathbf{X})\ar[r]&
        \cdots\\\cdots\ar[r]&
        \limn{s} H_{n+1}(\mathbf{X})\ar[r]&
        \overline{H}_{n-s+1}^{(s)}(\mathbf{X})\ar[r]&
        \overline{H}_{n-s+1}^{(s-1)}(\mathbf{X})\ar[r]&
        \limn{s+1} H_{n+1}(\mathbf{X})\ar[r]&
        \cdots
        
        \end{tikzcd}
    \end{center}
Moreover, this assignment determines a functor from pro-$\mathsf{Top}$ to the category of long exact sequences of abelian groups.
    \item[(iii)] \label{second_seq}
    For each integer $n$ and pro-topological space $\mathbf{X}$, there exists an exact sequence
    \begin{center}
        \begin{tikzcd}[column sep=small]
        0\ar[r] &\limn{1}_s\overline{H}_{n+1}^{(s)}(\Xbf)\ar[r]&\overline{H}_n(\Xbf)\ar[r]&\limn{}_s \overline{H}_n^{(s)}(\Xbf)\ar[r]&0.
        \end{tikzcd}
    \end{center}
Moreover, this assignment determines a functor from pro-$\mathsf{Top}$ to the category of short exact sequences of abelian groups.
\end{enumerate}
\end{thm}
The following two corollaries appear as Corollaries 1 and 2, respectively, in \cite{SHINA} and will feature in our proof of Theorem \ref{Om_add_cpt_supt}.
In the first of these, vanishing higher derived limits of intervals of homology pro-groups determine intervals on which the approximations $\overline{H}_{p}^{(s)}(\mathbf{X})$ to strong homology groups stabilize.

\begin{cor} \label{M_cor_1}
Let $1 \leq s_{0} \leq s_{1}$ be such that for a given integer $p$ one has
\[\limn{t} H_{p+s}(\mathbf{X})=0, \quad s_{0} \leq s \leq s_{1}, t>0.\]
Then the homomorphisms $j_{p}^{t-1, t}$ yield isomorphisms
\[
\overline{H}_{p}^{\left(s_{0}-1\right)}(\mathbf{X}) \approx \cdots \approx \overline{H}_{p}^{(s)}(\mathbf{X}) \approx \cdots \approx \overline{H}_{p}^{\left(s_{1}\right)}(\mathbf{X})
\]
for $s_{0} \leq s \leq s_{1}$.
\end{cor}

The second corollary records conditions under which the $s$-stage strong homology groups converge to the groups $\overline{H}_{p}(\mathbf{X})$.   

\begin{cor} \label{M_cor_2}
Let $s_{0} \geq 1$ be such that for a given integer $p$ one has
\[
\limn{t} H_{p+s}(\mathbf{X})=0, \quad s_{0} \leq s, t>0.
\]
Then the homomorphisms $j_{p}^{s-1, s}$ and $j_{p}^{s}:\overline{H}_p\to\overline{H}_p^{(s)}$ induce isomorphisms
\[
\overline{H}_{p}^{\left(s_{0}-1\right)}(\mathbf{X}) \approx \cdots \approx \overline{H}_{p}^{(s)}(\mathbf{X}) \approx \cdots \approx \overline{H}_{p}(\mathbf{X}), \quad s_{0} \leq s.
\]
\end{cor}

The last ingredient in our proof of Theorem \ref{Om_add_cpt_supt} is the existence of a Mayer-Vietoris exact sequence for strong homology. This follows from the fact that strong homology satisfies the Eilenberg-Steenrod axioms, but we record a more concrete proof.

\begin{lem}
Suppose that $X$ is paracompact and $U_1,U_2$ are closed subsets of $X$ such that $U_1\cup U_2=X$. 
Then there is a Mayer-Vietoris sequence for the triad $(X;U_1,U_2)$. 
\end{lem}

\begin{proof}
We assume $U_1\cap U_2\neq\emptyset$, as otherwise this is just finite additivity.
By \cite[Theorem 19.5]{SSH}, there are long exact sequences for the pairs $(X,U_1)$ and $(U_1,U_1\cap U_2)$.
Observe that since $U_1$ and $U_2$ are both closed in $X$, the map $U_1/U_1\cap U_2\to X/U_1$ is a closed map and, therefore, since it is also bijective and continuous, a homeomorphism.
By the strong excision property (\cite[Theorem 19.9]{SSH}), the map $k_*\colon\overline{H}_*(U_1,U_1\cap U_2)\to\overline{H}_*(X,U_1)$ is an isomorphism. Under these conditions, we generally obtain a Mayer-Vietoris sequence (see \cite[Theorem 15.3]{EileStee}).
\end{proof}

Having collected the necessary machinery, we now turn to the proof of Theorem \ref{Om_add_cpt_supt}. 

\begin{proof}[Proof of Theorem \ref{Om_add_cpt_supt}]
We show the bi-implications (\ref{Om_limn_add})$\,\Leftrightarrow\,$(\ref{cpt_add}) and (\ref{cpt_supt})$\,\Leftrightarrow\,$(\ref{cpt_add}).

First we will show that (\ref{Om_limn_add}) implies (\ref{cpt_add}).
For each $i$, fix a resolution $X_i\to\mathbf{X}_i$ such that $\mathbf{X}_i$ is an inverse sequence of compact polyhedra.
Let $\mathbf{X}$ be the induced resolution of $X=\coprod_iX_i$. For $\mathbf{C}$ an inverse system of spaces,
let $\overline{H}_p^{(s)}(\mathbf{C})$ denote the $s$-stage strong homology group described in Section 4 of \cite{SHINA}. 
\begin{claim}
$\overline{H}_p(X)=\overline{H}_p^{(1)}(\mathbf{X})$
and $\overline{H}_p(X_i)=\overline{H}_p^{(1)}(\mathbf{X}_i)$
\end{claim}
\begin{proof}
Observe that since $\mathbf{X}_i$ is a sequence, $\limn{s} H_n(\mathbf{X}_i)=0$ for any $n$ and $s\geq2$.
By (\ref{Om_limn_add}) together with the additivity of singular homology, $\limn{s}H_n(\mathbf{X})=0$.
When $\mathbf{C}=\Xbf$ or $\Xbf_i$, the section of the first Miminoshvili exact sequence given by
\begin{center}
    \begin{tikzcd}
    \limn{s}H_{p+s}(\mathbf{C})\ar[r]&\overline{H}_p^{(s)}(\mathbf{C})\ar[r]&\overline{H}_p^{(s-1)}(\mathbf{C})\ar[r]&\limn{s+1}H_{p+s}(\mathbf{C})
    \end{tikzcd}
\end{center}
implies that $H_p^{(s)}(\mathbf{C})\approx H_p^{(s-1)}(\mathbf{C})$ for $s\geq2$.
Application of the second Miminoshvili exact sequence completes the proof of the claim.
\end{proof}

In particular, (\ref{Om_limn_add}) will imply (\ref{cpt_add}) if it implies that the induced map $\bigoplus_i\overline{H}_p^{(1)}(\mathbf{X}_i)\to \overline{H}_p^{(1)}(\Xbf)$ is an isomorphism.
By the naturality of the first Miminoshvili exact sequence, there is a diagram with exact rows of the form

\begin{center}
    \begin{tikzcd}[column sep=small]
        0\ar[r]&\bigoplus_i\limn{1}H_{p+1}(\Xbf_i)\ar[d]\ar[r]&\bigoplus_i\overline{H}_p^{(1)}(\Xbf_i)\ar[r]\ar[d]&\bigoplus_i\overline{H}_p^{(0)}(\Xbf_i)\ar[r]\ar[d]&0\\
        0\ar[r]&\limn{1}H_{p+1}(\Xbf)\ar[r]&\overline{H}_p^{(1)}(\Xbf)\ar[r]&\overline{H}_p^{(0)}(\Xbf)\ar[r]&0
    \end{tikzcd}
\end{center}

By (\ref{Om_limn_add}), the map $\bigoplus_i\limn{1}H_{p+1}(\Xbf_i)\to \limn{1}H_{p+1}(\Xbf)$ is an isomorphism.
By the additivity of \v{C}ech Homology
(see \cite[Section 7]{SHINA}),
the map $\bigoplus_i\overline{H}_p^{(0)}(\Xbf_i)\to \overline{H}_p^{(0)}(\Xbf)$ is an isomorphism.
By the Five Lemma, we conclude that the map $\bigoplus_i\overline{H}_p^{(1)}(\Xbf_i)\to\overline{H}_p^{(1)}(\Xbf)$ is an isomorphism, completing the proof of the implication (\ref{Om_limn_add})$\,\Rightarrow\,$(\ref{cpt_add}). 

Next we will show that (\ref{cpt_add}) implies (\ref{Om_limn_add}). Let $\Gcal$ be an $\Om$-system of abelian groups. Fix for each $(n,k)$ a compact polyhedron $X_{n,k}$ which is a Moore space $M(G_{k,n},1)$ together with maps $f_{n,k}\colon X_{n,k}\to X_{n-1,k}$ such that the induced map on $H_1$ agrees with $p_{n,k}$.
Let $X_k=\lim_nX_{n,k}$ and $X=\coprod_kX_k$.
By the compactness of each $X_{n,k}$, each $X_k$ is compact and the map $X_k\to \Xbf_k=(X_{n,k},f_{n,k},\omega)$
is a resolution.
We also have, as usual, the induced resolution $X\to \Xbf$ in which $\Xbf$ is indexed by $\Om$.
The homology pro-groups of $\Xbf$ are the constant system
$\bigoplus_{k\in\omega}\Zbb$ in dimension $0$,
$\Gbf$ in dimension $1$, and $0$ in all other dimensions.
In particular, by Corollary \ref{M_cor_1} above, for each $p<0$ we have
$\overline{H}_p^{(0)}(\Xbf)\approx\overline{H}_p^{(1)}(\Xbf)\approx\ldots \approx\overline{H}_p^{(-p)}(\Xbf)$
and by Corollary \ref{M_cor_2} we have $\overline{H}_p^{(1-p)}(\Xbf)\approx\overline{H}_p^{(2-p)}(\Xbf)\approx\ldots\approx\overline{H}_p(\Xbf)$.
Similarly, 
$\overline{H}_p^{(1)}(\Xbf)\approx\overline{H}_p^{(2)}(\Xbf)\approx\ldots\approx\overline{H}_p(\Xbf)$ for all $p\geq0$.
Similar statements holds for each $\Xbf_i$.
In particular, since $\overline{H}_p^{(0)}$ is always additive and by (\ref{cpt_add})
an appropriate form of additivity holds for $\overline{H}_p$,
the natural map $\bigoplus_i(\overline{H}_p^{(s)}(\Xbf_i))\to\overline{H}_p^{(s)}(\Xbf)$
is an isomorphism for each integer $p$ and $s\geq 0$. Item (\ref{Om_limn_add})
now follows from an application of the Five Lemma to the following diagram with exact rows
obtained from the first Miminoshvili exact sequence 

\begin{center}
    \begin{tikzcd}[column sep=small]
        \bigoplus_i\overline{H}_{2-s}^{(s-1)}(\Xbf_i)\ar[r]\ar[d]&\bigoplus_i\overline{H}_{2-s}^{(s-2)}(\Xbf_i)\ar[d]\ar[r]&\bigoplus_i\limn{s}H_{1}(\Xbf_i)\ar[d]\ar[r]&\bigoplus_i\overline{H}_{1-s}^{(s)}(\Xbf_i)\ar[r]\ar[d]&\bigoplus_i\overline{H}_{1-s}^{(s-1)}(\Xbf_i)\ar[d]\\
        \overline{H}_{2-s}^{(s-1)}(\Xbf)\ar[r]&\overline{H}_{2-s}^{(s-2)}(\Xbf)\ar[r]&\limn{s}H_{1}(\Xbf)\ar[r]&\overline{H}_{1-s}^{(s)}(\Xbf)\ar[r]&\bigoplus_i\overline{H}_{1-s}^{(s-1)}(\Xbf_i)
        
    \end{tikzcd}
\end{center}

The implication (\ref{cpt_supt})$\,\Rightarrow\,$(\ref{cpt_add}) is not difficult.

By functoriality and finite additivity of $\overline{H}_p$, it suffices to show that the map $\bigoplus_{i\in\omega} \overline{H}_p(X_i)\to \overline{H}_p(\coprod_iX_i)$ is surjective. 
Fix $x\in \overline{H}_p(\coprod_iX_i)$ and, by compact supports, let $K$ be compact such that $x$ is in the image of the map $\overline{H}_p(K)\to\overline{H}_p(\coprod_iX_i)$. 
By compactness of $K$, there is a finite $F\subseteq\omega$ such that $K\subseteq\coprod_{i\in F}X_i$. 
Finite additivity implies that $\overline{H}_p(\coprod_{i\in F}X_i)\cong\bigoplus_{i\in F}\overline{H}_p(X_i)$. 
In particular, the map $\overline{H}_p(K)\to\overline{H}_p(\coprod_iX_i)$ factors through $\bigoplus_{i\in\omega} \overline{H}_p(X_i)$, so $x$ is in the image of the map $\bigoplus_{i\in\omega} \overline{H}_p(X_i)\to \overline{H}_p(\coprod_iX_i)$.

Finally, we will show that (\ref{cpt_add}) implies (\ref{cpt_supt}).
The template for our argument is Milnor's well-known proof of Lemma 1 in \cite{Milnor}.
Let $X$ be a locally compact separable metric space and fix open sets $\langle U_n\mid n\in\omega\rangle$ such that
\begin{itemize}
    \item  $U_n\subseteq U_{n+1}$ for each $n$, 
    \item  ${\overline{U}_n}$ is compact for each $n$, and
    \item $\bigcup_n U_n=X$.
\end{itemize}
Observe that the sequence of ${\overline{U}_n}$ is cofinal in all compact subsets of $X$.
Thus it suffices to show that $\dlim_n\overline{H}_p(U_n)=\overline{H}_p(X)$ for each $p$. Let $L\subseteq X\times[0,\infty)$ denote the space
\[L=\overline{U}_0\times[0,1]\cup\overline{U}_1\times[1,2]\cup\overline{U}_2\times[2,3]\cup...\]

\begin{claim}
The natural projection map $\pi\colon L\to X$ is a homotopy equivalence.
\end{claim}
\begin{proof}
We construct a homotopy inverse. Fix $\langle C_n\mid n\in\omega\rangle$ closed subsets of $X$ such that
\begin{itemize}
    \item $C_n\subseteq U_n$ for each $n$, 
    \item $C_n\subseteq C_{n+1}$ for each $n$, and
    \item $\bigcup_nC_n^\circ=X$.
\end{itemize}
To see that this is possible, let $C_n=\overline{\{x\in U_n\mid {B_{1/n}(x)}\subseteq U_n\}}$, for example. 

Using Urysohn's Lemma, fix continuous functions $f_n\colon X\to [0,1]$ such that
$f_n\upharpoonright C_n=0$ and $f_n\upharpoonright X\setminus U_n=1$.
Observe that for each $x$, $f_n(x)=0$ for all but finitely many $n$.
In particular, the sum $\sum_n f_n(x)$ is well-defined.
Moreover, $(\sum_n f_n)\upharpoonright C_k$, being the finite sum $f_0+\ldots+f_k$, is continuous.
Since the interiors of $C_k$ cover $X$, the sum $\sum_n f_n$ is continuous for any $k$.
We claim that the map $f\colon X\to L$ given by $x\mapsto(x,\sum_nf_n(x))$ is a homotopy inverse to the projection map. The composition $\pi\circ f$ is the identity and a homotopy between
$\operatorname{id}_L$ and $f\circ\pi$ is given by
\[(x,y,t)\mapsto(x,(1-t)y+t\sum_nf_n(x)).\]
\end{proof}
Let $L_0\subseteq L$ be the union of all ${\overline{U}_n}\times[n,n+1]$
with $n$ even. 
Similarly, Let $L_1\subseteq L$ be the union of all ${\overline{U}_n}\times[n,n+1]$ with $n$ odd. 
Observe that, by (\ref{cpt_add}) plus the homotopy axiom,
\[\overline{H}_{*}(L_{0}) \approx \overline{H}_{*}({\overline{U}_{0}}) \oplus \overline{H}_{*}({\overline{U}_{2}}) \oplus \overline{H}_{*}({\overline{U}_{4}}) \oplus \cdots,\]
with a similar assertion for $L_1$. Since $L_0\cap L_1$ is the disjoint union of the ${\overline{U}_n}\times[n+1]$, 
\[\overline{H}_{*}(L_0\cap L_{1}) \approx \overline{H}_{*}({\overline{U}_{0}})\oplus\overline{H}_{*}({\overline{U}_{1}}) \oplus \overline{H}_{*}({\overline{U}_{2}}) \oplus \overline{H}_{*}({\overline{U}_{3}}) \oplus \cdots.\]
Since $L_0$ and $L_1$ are closed in $L$ and $L$ is paracompact,
there is a Mayer-Vietoris sequence for the triad $(L;L_0,L_1)$.
The homomorphism
\[\psi\colon \overline{H}_*(L_0\cap L_1)\to \overline{H}_*(L_0)\oplus \overline{H}_*(L_1)\]
in this sequence is readily computed, and turns out to be the following:
\[\psi(a_0,a_1,...,0,0,...)=\left(a_0,p a_{1}+a_{2},p a_{3}+a_{4}, \cdots\right)\oplus\left(-pa_0-a_{1}, -p a_{2}-a_{3}, -p a_{4}-a_{5}, \cdots\right).\]
Here $a_n$ denotes a generic element of ${\overline{H}_*}({\overline{U}_n})$ and $p\colon \overline{H}_*({\overline{U}_n})\to\overline{H}_*({\overline{U}_{n+1}})$ denotes the inclusion homomorphism. It is convenient to precede $\psi$ by the automorphism $\alpha$ of
$\overline{H}_*(L_0\cap L_1)$ which multiplies each $h_n$ by $(-1)^{n}$. After shuffling the
terms on the right-hand side of the formula above, we obtain
\[
\psi \alpha(a_{0}, a_{1}, \cdots)=(a_{0}, a_{1}-p a_{0}, a_{2}-p a_{1}, a_{3}-p a_{2}, \cdots).
\]
It is clear from this expression that $\psi$ has kernel zero and has
cokernel isomorphic to the direct limit of the sequence $\{\overline{H}_*({\overline{U}_n})\}$. Now
the Mayer-Vietoris sequence
\[0 \longrightarrow \overline{H}_{*}\left(L_{0} \cap L_{1}\right) \stackrel{\psi}{\longrightarrow} \overline{H}_{*}\left(L_{0}\right) \oplus \overline{H}_{*}\left(L_{1}\right) \longrightarrow \overline{H}_{*}(L) \longrightarrow 0\]
completes the proof.
\end{proof}

As noted, in conjunction with Lemma \ref{add_lem}, the theorem just proven implies Theorem \ref{main_result}.

\end{document}